\documentclass[onefignum,onetabnum]{siamart220329}

\usepackage{graphicx} 
\usepackage{amsfonts}
\usepackage{mathtools}
\usepackage{url}
\usepackage[font=small]{caption} 
\usepackage{fancyvrb,upquote} 
\usepackage{cite} 
\usepackage{hyperref}

\newcommand{\R}{\mathbb{R}}

\renewcommand{\vec}[1]{\boldsymbol{#1}}		

\DeclareMathOperator{\vspan}{span}         

\DeclareMathAlphabet{\mathbf}{OT1}{cmr}{bx}{it}
\newcommand{\vb}{{\mathbf b}}
\newcommand{\ve}{{\mathbf e}}
\newcommand{\vh}{{\mathbf h}}
\newcommand{\vr}{{\mathbf r}}

\newcommand{\vv}{{\mathbf v}}
\newcommand{\vw}{{\mathbf w}}
\newcommand{\vx}{{\mathbf x}}
\newcommand{\vy}{{\mathbf y}}

\DeclareMathOperator{\cond}{cond}
\DeclarePairedDelimiter{\norm}{\lVert}{\rVert}      

\newcommand{\revone}[1]{#1}		
\newcommand{\revtwo}[1]{#1}		
\newcommand{\revthree}[1]{#1}		

\headers{A sketch-and-select Arnoldi~process}{S.~G\"uttel and I.~Simunec}

\title{A sketch-and-select Arnoldi~process}

\author{Stefan G\"uttel\thanks{Department of Mathematics, The University of Manchester, M13\,9PL Manchester, United Kingdom, \texttt{stefan.guettel@manchester.ac.uk}. S.\,G. acknowledges a fellowship from The Alan Turing Institute under the EPSRC grant EP/W001381/1 and a Royal Society Industry Fellowship IF/R1/231032.} \and Igor Simunec\thanks{Scuola Normale Superiore, 56126 Pisa, Italy,  \texttt{igor.simunec@sns.it}. \revtwo{The work of I.\,S. was funded in part by the Italian Ministry of University and Research under the PRIN2022 project ``Low-rank Structures and Numerical Methods in Matrix and Tensor Computations and their Application'', proposal code 20227PCCKZ -- CUP E53D23005520006.}}}

\date{}

\begin{document}

\maketitle

\begin{abstract}
    A sketch-and-select Arnoldi process to generate a  well-conditioned basis of a Krylov space at low cost is proposed. At each iteration the procedure utilizes randomized sketching to select a limited number of previously computed basis vectors to project out of the current basis vector. The computational cost grows linearly with the dimension of the Krylov space. The subset selection problem for the projection step is approximately solved with a number of heuristic algorithms and greedy methods used in statistical learning and compressive sensing.  
\end{abstract}



\section{Introduction}

The Arnoldi process \cite{Arnoldi1951} is a key component of many Krylov subspace methods for large-scale numerical linear algebra computations, including solving linear systems of equations and eigenvalue problems with  nonsymmetric matrices~$A\in\mathbb{R}^{N\times N}$; see, e.g., \cite{Saad2003, saad2011numerical}. The Arnoldi process is also used for solving least squares problems, approximating matrix functions or matrix equations, and in model order reduction, to name just a few other applications. Given a starting vector $\vb\in\mathbb{R}^N$ and an integer $m\ll N$, the Arnoldi process iteratively constructs an orthonormal basis $\{\vv_1,\vv_2,\ldots,\vv_m\}$ of the Krylov space $\mathcal{K}_m(A,\vb) := \vspan \{ \vb, A\vb, \ldots, A^{m-1}\vb\}$. More precisely, given $j$~orthonormal basis vectors $\vv_1 := \vb/\|\vb\|,\vv_2,\ldots,\vv_j$, the next basis vector is obtained by orthogonalizing $\vw_j := A\vv_j$ against all previous vectors,
\begin{equation}\label{eq:arn}
\widehat \vw_j := \vw_j - \sum_{i=1}^j h_{i,j} \vv_i, \quad h_{i,j} := \vv_i^T \vw_j,
\end{equation}
and then setting $\vv_{j+1}:=\widehat \vw_j/h_{j+1,j}$ with $h_{j+1,j}= \|\widehat \vw_j\|$. Collecting the basis vectors into $V_m = [\vv_1,\vv_2,\ldots,\vv_m]\in\mathbb{R}^{N\times m}$ and the orthogonalization coefficients into $H_m = [h_{i,j}]\in\mathbb{R}^{m\times m}$, the Arnoldi process generates an Arnoldi decomposition 
\[
    A V_m = V_m H_m + h_{m+1,m} \vv_{m+1} \ve_m^T,
\]
where $\ve_m\in\mathbb{R}^m$ denotes the $m$-th canonical unit vector. By construction, $H_m$ is an upper-Hessenberg matrix. 

In terms of arithmetic cost, the Arnoldi process requires $m$ matrix-vector products~$A \vw_j$, as well as $m(m+1)/2$ inner products and vector operations (``\texttt{axpy}'' in BLAS-1 naming), for a total of 
\[
O(m\cdot \texttt{nnz}(A) + N m^2)
\]
arithmetic operations. For sufficiently sparse $A$ and as $m$ increases, this cost will be dominated by the $N m^2$-term for the orthogonalization. There are at least two possible ways to reduce this cost. The first one is to restart the Arnoldi process after $m$ iterations, using $\vv=\vv_{m+1}$ as the starting vector for the next cycle. Restarting  is particularly natural in the context of solving linear systems of equations (as there exists a linear error equation $A\ve=\vr$ where $\vr$ is the residual), but it can also be used for eigenvalue problems \cite{Morgan1996,Stewart01,StathopoulosEtAl1998} or matrix function computations \cite{AfanasjewEtAl2008a,FrommerGuettelSchweitzer2014a}.  Of course, the combined Krylov basis computed after $\ell>1$ restarts is no longer orthonormal and this usually leads to a delayed convergence in restarted Krylov methods. \revone{However, there exist deflation approaches and  recycling methods which aim to mitigate the convergence delay; see \cite{soodhalter2020survey} for a survey.} \revthree{We mention that restarted Krylov methods fall within the more general category of Krylov methods with non-orthogonal bases, which have been studied in \cite{SimonciniSzyld2005}.}

The second, more recently proposed approach to reduce the arithmetic cost of the Arnoldi process is to employ randomized sketching;  see, e.g., \cite{balabanov2019randomized,balabanov2021randomized,balabanov2021randomizedblock,nakatsukasa2021fast,balabanov2022randomized}. 
The key tool of sketching is an embedding matrix $S\in\mathbb{R}^{s\times N}$ with $m < s\ll N$ that distorts the Euclidean norm $\|\,\cdot\,\|$ of vectors in a controlled manner~\cite{sarlos2006improved,woodruff2014sketching}. More precisely, given a positive integer $m$ and some $\varepsilon\in [0,1)$, we assume that $S$ is such that for all vectors~$\vv$ in the Krylov space $\mathcal{K}_{m+1}(A,\vb)$,
\begin{equation}
(1-\varepsilon) \| \vv \|^2 \leq \| S \vv\|^2 \leq (1+\varepsilon) \|\vv\|^2.
\label{eq:sketch}
\end{equation}
The matrix $S$ is called an \emph{$\varepsilon$-subspace embedding} for $\mathcal{K}_{m+1}(A,\vb)$. 
Condition~\cref{eq:sketch} can equivalently be stated with the Euclidean inner product~\cite[Cor.~4]{sarlos2006improved}. 
In practice, such a matrix~$S$ is not explicitly available and we hence have to draw it at random to achieve~\cref{eq:sketch} with high probability. There are several ways to construct a random matrix $S$ with this property, see for instance the discussions in \cite[Sec.~2.3]{nakatsukasa2021fast} or \cite[Sec.~2.1]{balabanov2022randomized}. 

\revtwo{
In Krylov subspace methods for linear algebra problems, such as the GMRES method for linear systems \cite{SaadSchultz1986}, the orthogonality of the Krylov basis is exploited to compute the solution of the problem projected on the Krylov space efficiently.
If we forego the assumption that the Krylov basis is orthonormal, it is still possible to obtain an approximate solution to the projected problem by means of randomized sketching. 
This strategy allows us to replace the Arnoldi process with a cheaper procedure that constructs a non-orthogonal basis, such as truncated Arnoldi.
This kind of approach was first proposed in \cite{nakatsukasa2021fast} for the solution of linear systems of equations and Rayleigh--Ritz extraction of eigenvalues, and it was later extended in \cite{GS23,cortinovis2022speeding} for matrix function computations.

As a motivating example for our work, we briefly describe the sketched GMRES (sGMRES) procedure from \cite{nakatsukasa2021fast}.
Consider the linear system $A \vx = \vb$, with $A \in \R^{N \times N}$ and $\vx, \vb \in \R^{N}$. Assume that we have constructed a basis $V_m$ of the Krylov space~$\mathcal{K}_m(A, \vb)$ that is not necessarily orthonormal, and let $S \in \R^{s \times N}$ be a sketching matrix.
Recall that in its $m$-th iteration, GMRES computes the solution to the $N \times m$ least squares problem
\begin{equation*}
    \vy_m := \underset{\vy \in \R^m}{\arg\min} \norm{A V_m \vy - \vb }_2
\end{equation*}
and the associated approximate solution $\vx_m := V_m \vy_m$ to the linear system. In the case where $V_m$ is generated by the standard Arnoldi method, the least squares problem can be solved efficiently by exploiting the Arnoldi relation and the orthogonality of~$V_m$.

On the other hand, the sGMRES algorithm solves the $s \times m$ sketched least squares problem
\begin{equation*}
    \vy_m^S := \underset{\vy \in \R^m}{\arg\min} \norm{S ( A V_m \vy - \vb )}_2,
\end{equation*}
and computes the approximate solution $\vx_m^S := V_m \vy_m^S$. If $S$ is an $\varepsilon$-embedding for $\vspan([A V_m, \vb])$, it is easy to see that 
\begin{equation*}
    \norm{A \vx_m^S - \vb}_2 \le \frac{1 + \varepsilon}{1-\varepsilon} \norm{A \vx_m - \vb}_2.
\end{equation*}
By using an embedding dimension $s = 2(m+1)$, we typically obtain an $\varepsilon$-embedding with $\varepsilon = 1/\sqrt{2}$, which causes the sGMRES residual norm to be bounded from above by roughly $6$ times the GMRES residual norm with high probability.
We refer the reader to \cite[Section~3]{nakatsukasa2021fast} for further information on sGMRES and its implementation. 

By taking $s = O(m)$ and using an efficient subspace embedding such as the subsampled random cosine or Fourier transform~\cite{sorensen1993efficient,WoolfeLibertyRokhlinTygert2008}, requiring $O(N\log s)$ operations when applied to a single vector, the $s \times m$ sketched least squares problem in sGMRES can be solved with $O(Nm \log m + m^3)$ arithmetic operations. If we use a non-orthogonal basis $V_m$ constructed with a fast procedure, the resulting method can be more efficient than GMRES with a fully orthogonal basis, the latter requiring at least $O(Nm^2)$ arithmetic operations. 

}

One of the most straightforward approaches to efficiently generate the Krylov basis with a reduced number of projection steps\footnote{Mathematically, the term ``orthogonalization'' is no longer adequate when the basis is non-orthogonal, so we use the term projection step to refer to a vector operation in truncated Arnoldi.} is the \emph{truncated Arnoldi procedure.} Let a truncation parameter $k$ be given, then in place of \cref{eq:arn} we use the iteration 
\begin{equation*}
\widehat \vw_j := \vw_j - \sum_{i=\max\{ 1, j+1-k\}}^j h_{i,j} \vv_i, \quad h_{i,j} := \vv_i^T \vw_j.
\end{equation*}
Alternatively, this same truncated Arnoldi procedure can  be combined with sketching by replacing the coefficients $h_{i,j}$ by their sketched counterparts $\widetilde h_{i,j} := (S\vv_i)^T (S\vw_j)$. A~key benefit of truncation is that the $O(Nm^2)$ cost of the orthogonalization is reduced to $O(Nmk)$, i.e., it grows linearly with the Krylov basis dimension~$m$. 

The truncated Arnoldi procedure is likely inspired by the Lanczos process for a symmetric matrix~$A$, which is mathematically equivalent to truncated Arnoldi with $k=2$. The Faber--Manteuffel theorem \cite{faber1984necessary} gives a complete characterization of the matrices~$A$  for which there is a short-term recursion that generates an orthogonal set of Krylov basis vectors. The truncated Arnoldi procedure first appeared in the context of eigenvalue problems~\cite[Sec.~3.2]{Saad1980b} and linear systems~\cite[Sec.~3.3]{Saad1981}. For computations involving matrix functions, truncated Arnoldi has been used for the matrix exponential in~\cite{GaudreaultRainwaterTokman2018,Koskela2015} and for more general matrix functions in~ \cite{GS23,cortinovis2022speeding}. We also refer the reader to~\cite[Sec.~4]{nakatsukasa2021fast} and~\cite[Sec.~5]{PhilippeReichel2012}. 

\revtwo{An alternative strategy to construct a Krylov basis for sGMRES is to modify the Arnoldi algorithm in order to generate an orthonormal sketched basis, an approach that has been proposed and applied in \cite{balabanov2019randomized, balabanov2021randomized, balabanov2021randomizedblock}.}
\revtwo{More precisely,} the inner products computed in~\cref{eq:arn} \revtwo{are replaced} by inner products on sketched vectors
\begin{equation}\label{eq:arns}
\widehat \vw_j := \vw_j - \sum_{i=1}^j \widetilde h_{i,j} \vv_i, \quad \widetilde h_{i,j} := (S\vv_i)^T (S\vw_j).\nonumber
\end{equation}
Effectively, the process then computes an orthonormal sketched basis $S V_{m+1}$. Using an efficient subspace embedding that requires $O(N \log s)$ operations when applied to a single vector, the overall complexity is now  
\[
O(m\cdot \texttt{nnz}(A) + mN\log s+ N m^2).
\]
\revone{Even though the cost still grows quadratically with~$m$, this approach lowers the cost of computing all $(m+1)m/2$ inner products from $N (m+1)m/2$ to $s (m+1)m/2$, reducing the total number of flops roughly by a factor of~$2$ and requiring less communication compared to Arnoldi with full orthogonalization.}
It follows from~\cite[Cor.~2.2]{balabanov2022randomized} that
\begin{equation}\label{eq:basiscond}
    \bigg(\frac{1-\varepsilon}{1+\varepsilon}\bigg)^{1/2} \cond(SV_{m+1})  \leq \cond(V_{m+1}) 
    \leq
    \bigg(\frac{1+\varepsilon}{1-\varepsilon}\bigg)^{1/2} \cond(SV_{m+1}),
\end{equation}
so the computed Krylov basis $V_{m+1}$ will be close to orthonormal provided that $\varepsilon$ is sufficiently small (i.e., $s$ is sufficiently large which, in practice, means choosing~$s$ between $2m$ and $4m$). \revtwo{This basis generation method does not provide the same speedup as truncated Arnoldi, but it has great stability properties due to the guarantees on the condition number of $V_{m+1}$.}

\revone{It has been observed in \cite{nakatsukasa2021fast} that sGMRES has a convergence behavior similar to GMRES as long as the condition number of the Krylov basis~$V_m$ remains bounded away from the inverse of the machine precision.}
Unfortunately, \revone{for several problems} the Krylov bases generated by truncated Arnoldi (with and without sketching) can become severely ill-conditioned even for moderate~$m$, \revone{limiting the applicability of sketched Krylov subspace methods}. 
\revone{The goal of this paper is to provide a procedure to construct a Krylov basis $V_m$ with the same $O(Nmk)$ runtime as truncated Arnoldi which reduces as much as possible the growth of the condition number of~$V_m$. The hope is that this will increase the number of problems to  which methods such as sketched GMRES can be applied  successfully.

We propose and test an alternative approach \revone{to truncated Arnoldi} which we call the \emph{sketch-and-select Arnoldi process.} The key idea is simple: instead of projecting each new Krylov basis vector against the $k$ previous basis vectors, we use the sketched version of the Krylov basis to identify $k$~candidates for the projection. \revone{The selection of basis vectors for the projection step is framed in terms of sparse least squares approximation, a problem that has been extensively studied in statistical learning and compressive sensing.} 
We then demonstrate with performance profiles that the sketch-and-select Arnoldi process with a simple select strategy to determine the candidate vectors often leads to much better conditioned Krylov bases, outperforming truncated Arnoldi and many other tested methods. \revone{We test the performance of sketch-and-select Arnoldi as a basis constructor within sGMRES, showing that compared to truncated Arnoldi it frequently allows us to attain a moderately higher accuracy before the basis becomes too ill-conditioned.} \revtwo{On the other hand, we also provide numerical evidence for situations where the basis condition number does not seem to be the key quantity controlling the numerical behaviour of sGMRES, leading to interesting open questions for future work.}

}

\section{The sketch-and-select Arnoldi process}\label{sec:ssa}

At iteration $j$ of the sketch-and-select Arnoldi process we compute $\vw_j := A\vv_j$ as in the standard Arnoldi process, but then aim to determine coefficients $h_{i,j}$ ($i=1,\ldots,j$) of which at most $k$ are nonzero. We then compute 
\begin{equation}\label{eq:arnss}
\widehat \vw_j := \vw_j - \sum_{h_{i,j}\neq 0} h_{i,j} \vv_i
\end{equation}
and set $\vv_{j+1}:=\widehat \vw_j/h_{j+1,j}$ with a suitable scaling factor~$h_{j+1,j}$. To obtain the nonzero coefficients $h_{i,j}$ we use the sketched Krylov basis $SV_j$ and the sketched vector $S\vw_j$. For any iteration $j>k$, we \revone{propose to select the index set $I$ of the nonzero coefficients $h_{i,j}$ by approximately solving} the following sparse least squares problem: select an index set $I \subseteq \{ 1,2,\ldots,j \}$ with $|I| = k$ as 
\[
   \underset{I}{\arg\min} \underset{\vh\in\mathbb{R}^{|I|}}{\min} \| S\vw_j - SV_j(\,:\,,I) \vh  \|.
\]
Here we have used MATLAB notation to denote column selection. Given the index set~$I$, the components of the best $\vh$ are then used as the projection coefficients~$h_{i,j}$ in~\cref{eq:arnss}. Finally, the scaling factor $h_{j+1,j}$ is chosen so that $\|S \vv_{j+1}\| = 1$.

The determination of an optimal index set $I$ is also known as best subset selection problem, a classic topic of model selection in statistical learning \cite{miller2002subset}.  There are two main variants of this problem; (i) the problem considered above where the sparsity level~$k$ is prescribed, and (ii) for a given tolerance $\epsilon >0$, the problem of selecting an index set~$I$ with fewest possible elements so that $\min_{\vh\in\mathbb{R}^{|I|}} \| S\vw_j - SV_j(\,:\,,I) \vh  \|\leq \epsilon$. It is known that the determination of a global minimizer for such problems is NP-hard; see~\cite{natarajan1995sparse}. Nevertheless, a vast amount of literature has been devoted to developing efficient optimization algorithms for this task. A review of these methods is beyond the scope of this paper, so we refer to the excellent overview in~\cite{bertsimas}.

A simple approach to selecting the index set~$I$ is to  retain the  $k$ components of $\vh := (SV_j)^\dagger (S\vw_j)$ which are largest in modulus, ignoring the remaining $j-k$ components. We found this to perform very well in the  experiments reported in \cref{sec:numex}, and we present a basic MATLAB implementation in \cref{alg:ssa}. We refer to this variant as \texttt{sketch + select pinv} (for ``pseudoinverse'').

Another approach is to select $I$ as above, but then to recompute the corresponding coefficients as $\vh = \revone{(}SV_j(\,:\,,I)\revone{)}^\dagger (S\vw_j)$. This ensures that the projected $S\widehat \vw_j$ is orthogonal to $\vspan(SV_j(\,:\, ,I))$. We refer to this variant as \texttt{sketch + select pinv2} (as two pseudoinverses are computed).

In the variant \texttt{sketch + select corr} (for ``correlation'') we select $I$ as the components of $(SV_j)^T (S\vw_j)$ which are largest in modulus, using the $k$ inner products as the projection coefficients $h_{i,j}$.\footnote{The \texttt{sketch + select corr}  variant has been hinted at in \cite[version~1, Sec.~4.3]{nakatsukasa2021fast}: ``Indeed, $(\mathbf S\mathbf q_i)^* (\mathbf \mathbf S\mathbf A\mathbf q_j) \approx \mathbf q_i^* \mathbf A\mathbf q_j$ for all $i \leq j$, so we can choose to orthogonalize $\mathbf A\mathbf q_j$ only against the basis vectors $\mathbf q_i$ where the inner product is nonnegligible.''}  
Alternatively, we can recompute the projection coefficients as $\vh = \revone{(}SV_j(\,:\,,I)\revone{)}^\dagger (S\vw_j)$, referred to as variant \texttt{sketch + select corr pinv}. 

Finally, we also test three popular methods for sparse approximation, namely orthogonal matching pursuit (OMP)~\cite{pati1993orthogonal,tropp2004greed}, subspace pursuit (SP)~\cite{dai2009subspace}, and the ``Algorithm Greedy'' from \cite{natarajan1995sparse}. We have chosen these methods due to their popularity but also because they naturally allow for a fixed sparsity level~$k$, as opposed to, e.g., LASSO~\cite{tibshirani1997lasso}.

\revtwo{
OMP is a greedy method that, for each Arnoldi iteration~$j$, \revone{selects the $k$ columns of $SV_j$ for the projection sequentially.} \revone{OMP runs for $k$ iterations,} with the $i$-th inner iteration involving \revone{the computation of $j$ sketched inner products,}  the solution of an $s \times i$~least squares problem with the currently selected columns of $S V_j$ \revone{and the computation of a matrix-vector product with the $s \times i$ matrix with the currently selected columns}. \revone{Since the columns are added in sequence, one in each \revone{inner} iteration, the $s \times i$ least squares problem in the $i$-th OMP iteration can be solved efficiently by updating a QR factorization of the matrix containing the columns selected up to the $i$-th iteration. We refer to \cite[Section~II.C]{tropp2004greed} for a complete description of the OMP algorithm.} 

SP\revone{ \cite[Algorithm~1]{dai2009subspace}} is also an iterative method \revone{which, in contrast to OMP, selects the $k$ indices for projection by examining sets of $k$ columns at a time: in each iteration, the current selection is updated by considering $k$ additional candidate columns, and retaining only $k$ of the~$2k$ available columns according to a certain criterion.} We have fixed the number of SP iterations to~1, \revone{since this already yields good results}. As a result, the operations performed by SP in the $j$-th Arnoldi iteration are \revone{the computation of $2j$ sketched inner products,} a matrix-vector product with an $s \times k$ matrix, and the solution of two $s \times k$ least squares problems, as well as a least squares problem with a matrix of size at most $s \times 2k$ formed with selected columns of~$S V_j$.  

The ``Algorithm Greedy'' from \cite{natarajan1995sparse} is an iterative method that requires $k$ iterations for each Arnoldi iteration~$j$, and each inner iteration involves the computation of~$2j$ inner products between sketched vectors, for a total cost of $2jk$ sketched inner products, and the solution of one $s \times k$ least squares problem in the $j$-th Arnoldi iteration. \revone{See \cite[Algorithm Greedy]{natarajan1995sparse} for a pseudocode of this algorithm.}

In terms of computational cost, all sketch-and-select variants only perform operations on  small sketched matrices and vectors to determine the projection coefficients. Hence these computations are comparably cheap, but \texttt{sketch + select corr} and \texttt{sketch + select pinv} are the cheapest methods, as  in the $j$-th Arnoldi iteration they only require, respectively, the computation of a matrix-vector product with a matrix of size $s \times j$ and the solution of a single $s\times j$ least squares problem. 
}

\begin{figure}[ht]
\begin{Verbatim}[frame=single]
function [V, H, SV, SAV] = ssa(A, b, m, k)
%% Sketch-and-select Arnoldi process

sw = sketch(b); nsw = norm(sw);
SV(:,1) = sw/nsw; V(:,1) = b/nsw; H = [];
for j = 1:m
    w = A*V(:,j); 
    sw = sketch(w); SAV(:,j) = sw; 
    % the following two lines perform the select operation
    coeffs = SV(:,1:j) \ sw;
    [~,ind] = maxk(abs(coeffs),k);
    w = w - V(:,ind)*coeffs(ind);
    sw = sw - SV(:,ind)*coeffs(ind);
    nsw = norm(sw); 
    SV(:,j+1) = sw/nsw; V(:,j+1) = w/nsw;
    H(ind,j) = coeffs; H(j+1,j) = nsw;
end
\end{Verbatim}
\vspace*{-1mm}
\caption{Basic (non-optimized) MATLAB implementation of the sketch-and-select Arnoldi process. The function uses a \texttt{sketch} function that takes as input a vector with $N$ components and returns a sketch with $s\ll N$ components. In this variant, which we refer to as \texttt{sketch + select pinv}, the indices of Krylov basis vectors to project out are obtained by keeping the $k$ coefficients of the orthogonal projection which are largest in modulus. In an efficient implementation, the least squares problems with the sketched basis $S V_j$ should be solved by retaining and updating at each iteration a QR factorization of $S V_j$.
\label{alg:ssa}}
\end{figure}

\section{Numerical experiments}\label{sec:numex}

We now test variants of the proposed sketch-and-select Arnoldi process on a range of matrices from the SuiteSparse Matrix Collection (formerly the University of Florida Sparse Matrix Collection \cite{davis2011university}), a widely used set of sparse matrix benchmarks collected from a wide range of applications. We include 80 matrices $A$ in our test, which correspond to all square numerically nonsymmetric matrices in the collection (as of June 2023) with sizes between $N=10^4$ and $10^6$. The starting vector~$\vb$ is chosen at random with unit normally distributed entries and kept constant for all tests with the same matrix dimension. 
The MATLAB scripts to reproduce the experiments in this section are available at \url{https://github.com/simunec/sketch-select-arnoldi}.

In the tests below we compare the seven variants of the sketch-and-select Arnoldi process introduced in \cref{sec:ssa}, each using a different method for the sparse least squares problem, as well as the truncated Arnoldi process with and without sketching.

\subsection{Illustration with a single matrix}
\label{subsec:experiments--illustration-with-single-matrix}
For our first experiment we plot in \cref{fig:test1} the condition number $\cond(V_m)$ as a function of $m$ for the SuiteSparse problem  \texttt{Norris/torso3}, a matrix of size $N=259,156$ corresponding to a finite difference electro-physiological 3D model of a torso. We use a truncation parameter of $k=2$ and $k=5$ and perform $m=100$ Arnoldi iterations. The sketching operator is the subsampled random Hadamard transform (SRHT) \cite[Sec.~1.3]{balabanov2021randomizedblock} with an embedding dimension of~$s=200$. 
\revone{Very similar results were obtained by running the same experiments with a subsampled random cosine transform as a sketching operator.}

We see from \cref{fig:test1} that most sketch-and-select Arnoldi variants exhibit a much smaller condition number growth than truncated Arnoldi. The main exception is \texttt{sketch + select corr} which performs very badly in both cases, and \texttt{sketch + select corr pinv} which is only acceptable for $k=5$. Surprisingly, the variant \texttt{sketch + select pinv2} breaks down after about $m=80$ Arnoldi iterations in the case $k=2$, but it works very well for $k=5$. The four most reliable variants are \texttt{sketch + select pinv}, \texttt{sketch + select OMP}, \texttt{sketch + select SP}, and \texttt{sketch + select greedy}, all leading to a rather slow growth of the condition number. 

\begin{figure}[h!]
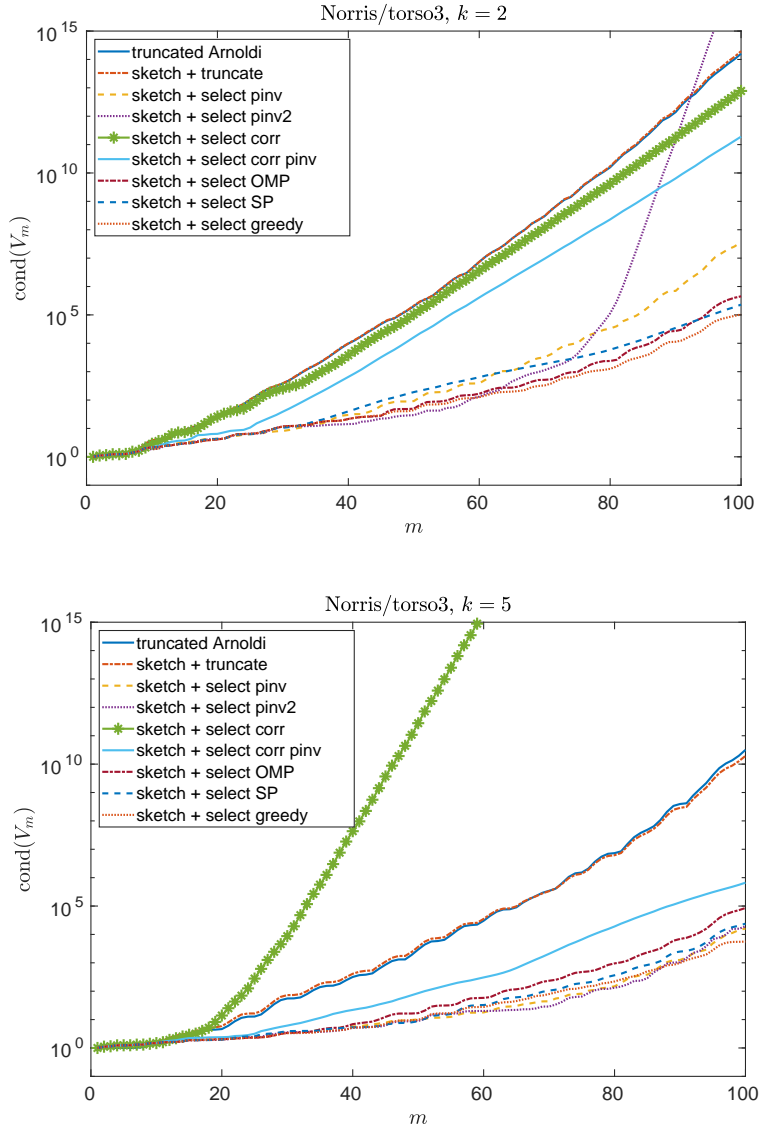

    \centering
    \includegraphics[scale=0.4]{test1_1_m100_s2_k2.pdf}
    \bigskip%
    
    \includegraphics[scale=0.4]{test1_1_m100_s2_k5.pdf}
    \caption{Basis condition number growth of nine different methods for the matrix \texttt{Norris/torso3}, using a truncation parameter of $k=2$ (top) and $k=5$ (bottom).}
    \label{fig:test1}
\end{figure}

\subsection{Performance profiles}
Our next tests involve all 80~matrices of the SuiteSparse Matrix Collection and we use performance profiles \cite{dolan2002benchmarking} to visualize the results. In a performance profile, each algorithm is represented
by a non-decreasing curve in a $\theta$–$y$ graph. The $\theta$-axis represents a tolerance
$\theta\geq 1$ and the $y$-axis corresponds to a fraction $y\in [0, 1]$. If a curve passes
through a point $(\theta,y)$ this means that the corresponding algorithm performed within
a factor $\theta$ of the best observed performance on $100 y$\% of the test problems. For
$\theta  = 1$ one can read off on what fraction of all test problems each algorithm was
the best performer, while as $\theta \to \infty$ all curves approach the value $y = 1$, unless
an algorithm has failed on a fraction of the test problems.

For each test problem, we run each algorithm until the condition number of the constructed basis becomes larger than $10^{12}$, up to a maximum basis dimension. The performance ratio is computed as the inverse of the basis dimension that is reached, so that, e.g., $\theta=2$ would correspond to an algorithm that generates a Krylov basis with half the size of the basis generated by the best algorithm.

The top panel in \cref{fig:test2a} shows the performance profiles for a target basis condition number of $10^{12}$ and a maximum basis dimension of $m = 100$, with a truncation parameter of $k=2$. The embedding dimension is $s = 200$. 
The four most reliable variants are \texttt{sketch + select pinv}, \texttt{sketch + select OMP}, \texttt{sketch + select SP}, and \texttt{sketch + select greedy}, and they are almost indistinguishable in performance. \revtwo{Among those methods, \texttt{sketch + select pinv} is the most straightforward to implement and the most computationally efficient.} The bottom panel in \cref{fig:test2a} displays the dimensions of the Krylov bases constructed for each test matrix by the four best performing algorithms and by truncated Arnoldi. The matrices are sorted so that the dimensions of the bases generated by truncated Arnoldi are in non-decreasing order.

\begin{figure}[h!]
    \centering
    \includegraphics[scale=0.4]{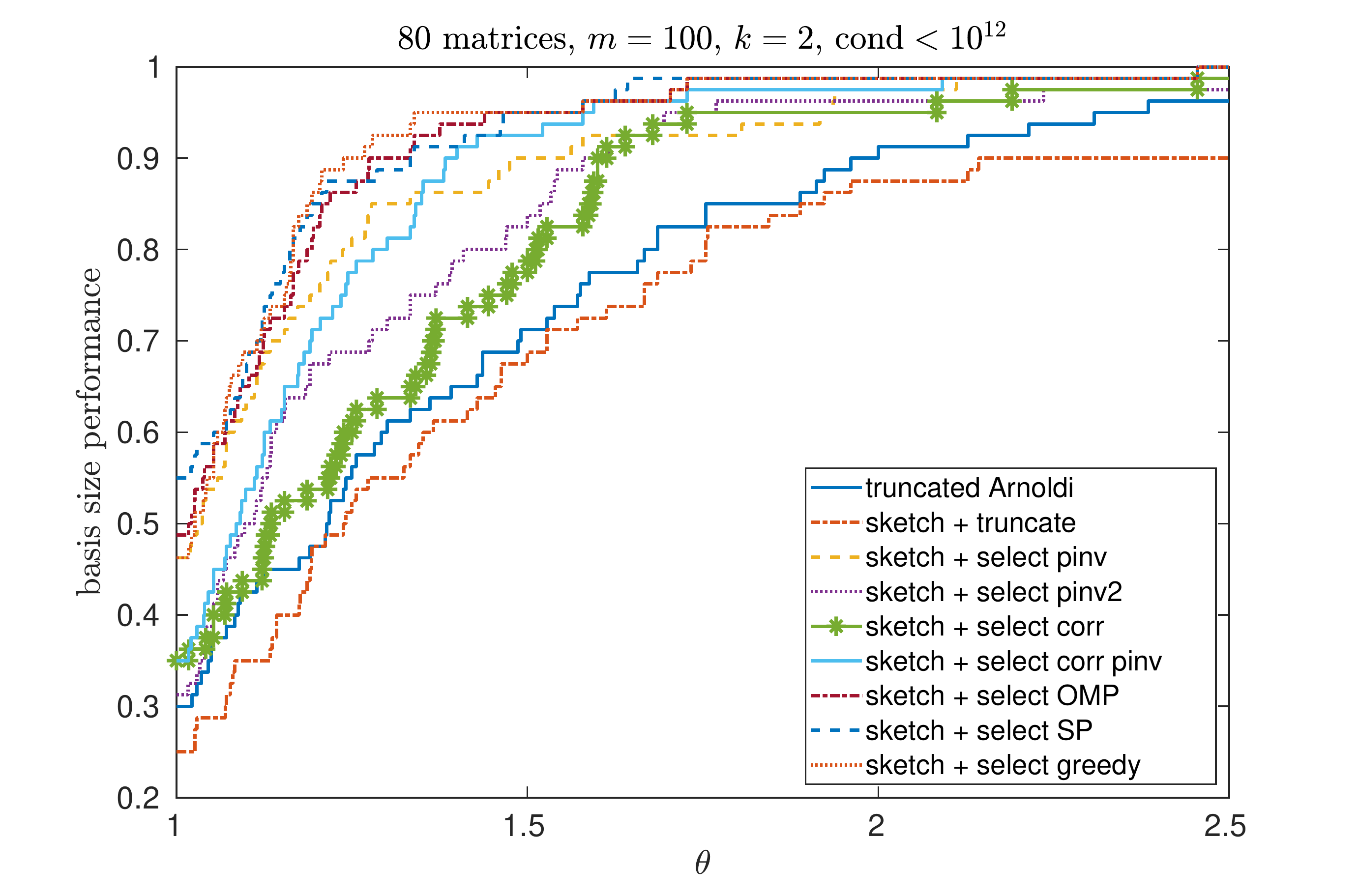}
    \bigskip%
    
    \includegraphics[scale=0.4]{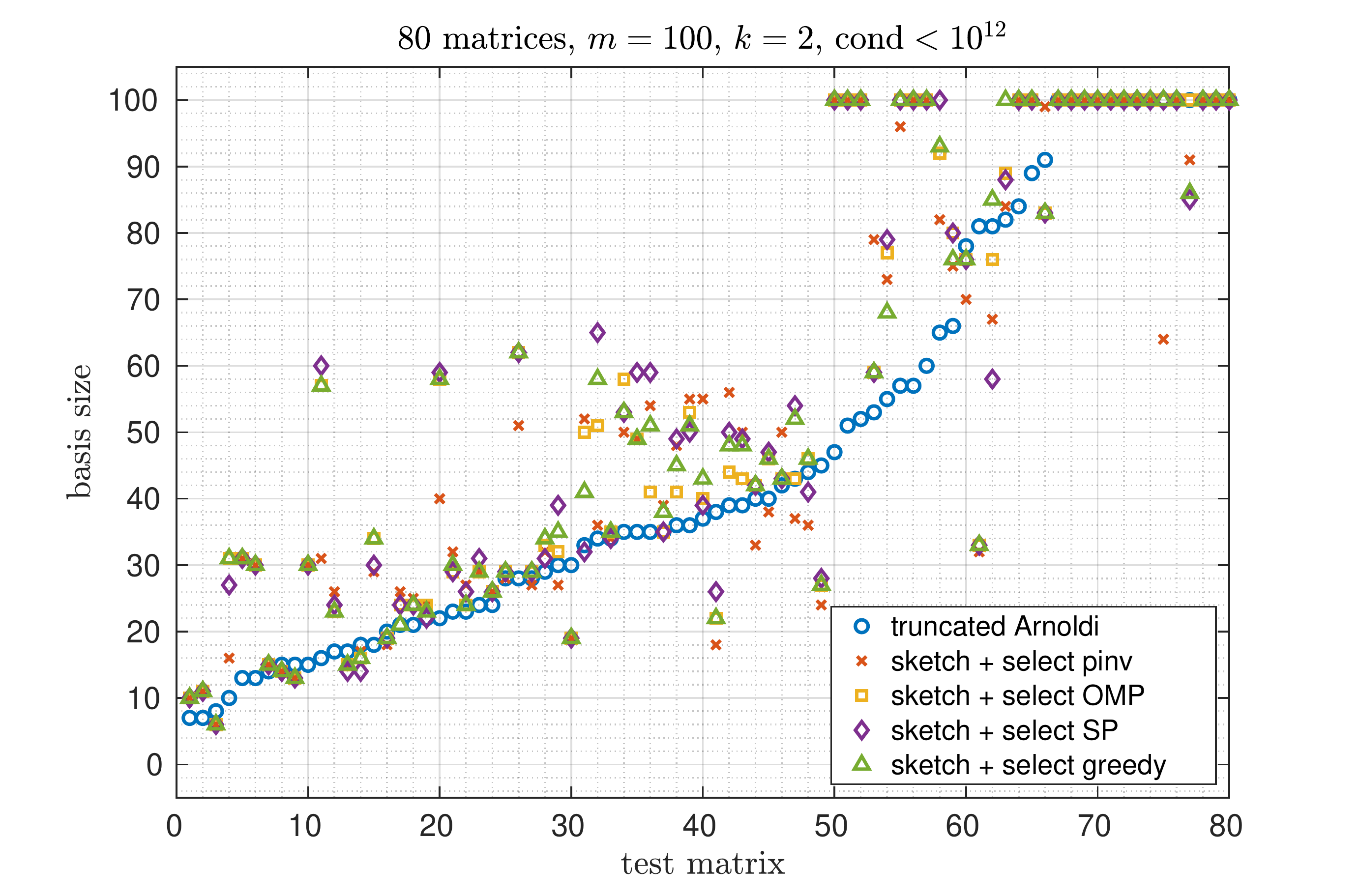}
    \caption{Basis sizes reached by the different methods, using a truncation parameter of $k=2$, with maximum basis size $m = 100$ and condition number bounded by $10^{12}$. Performance profiles (top) and basis sizes for the best performing methods (bottom).}
    \label{fig:test2a}
\end{figure}

A similar picture emerges in \cref{fig:test2b,fig:test2c}, where we have increased the truncation parameter to $k=5$ and $k=10$, respectively; the maximum basis dimension is increased  to $m = 150$ and $m = 200$, respectively, and the embedding dimension is chosen as $s = 2m$. With these parameters, the difference in performance between truncated Arnoldi and the best \texttt{sketch + select} variants becomes more significant, and the \texttt{sketch + select pinv} variant can be seen to perform slightly better than \texttt{sketch + select OMP}, \texttt{sketch + select SP} and \texttt{sketch + select greedy}.

\begin{figure}[h!]
    \centering
    \includegraphics[scale=0.4]{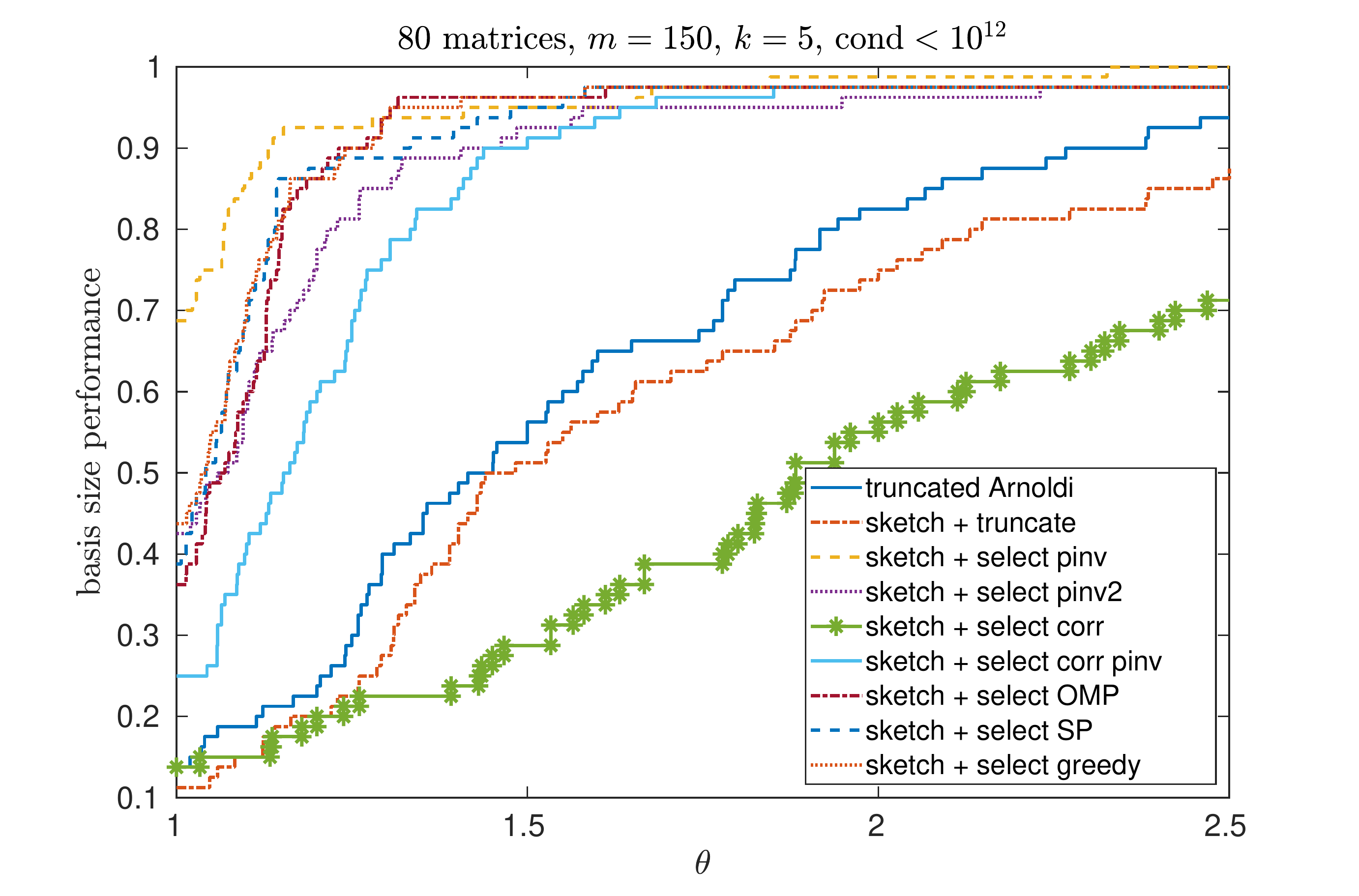}
    \bigskip%
    
    \includegraphics[scale=0.4]{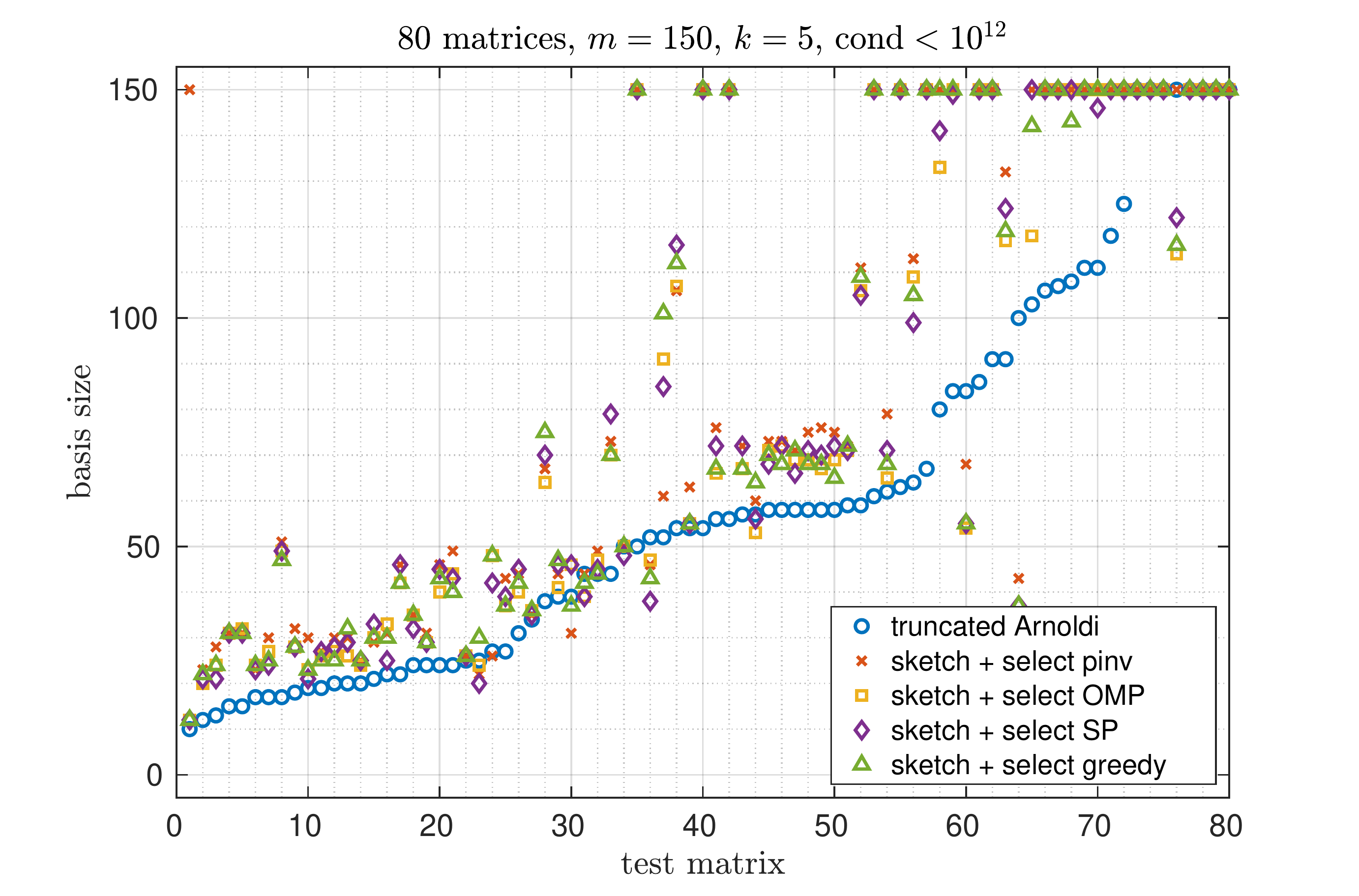}
    \caption{Basis sizes reached by the different methods, using a truncation parameter of $k=5$, with maximum basis size $m = 150$ and condition number bounded by $10^{12}$. Performance profiles (top) and basis sizes for the best performing methods (bottom).}
    \label{fig:test2b}
\end{figure}

\begin{figure}[h!]
    \centering
    \includegraphics[scale=0.4]{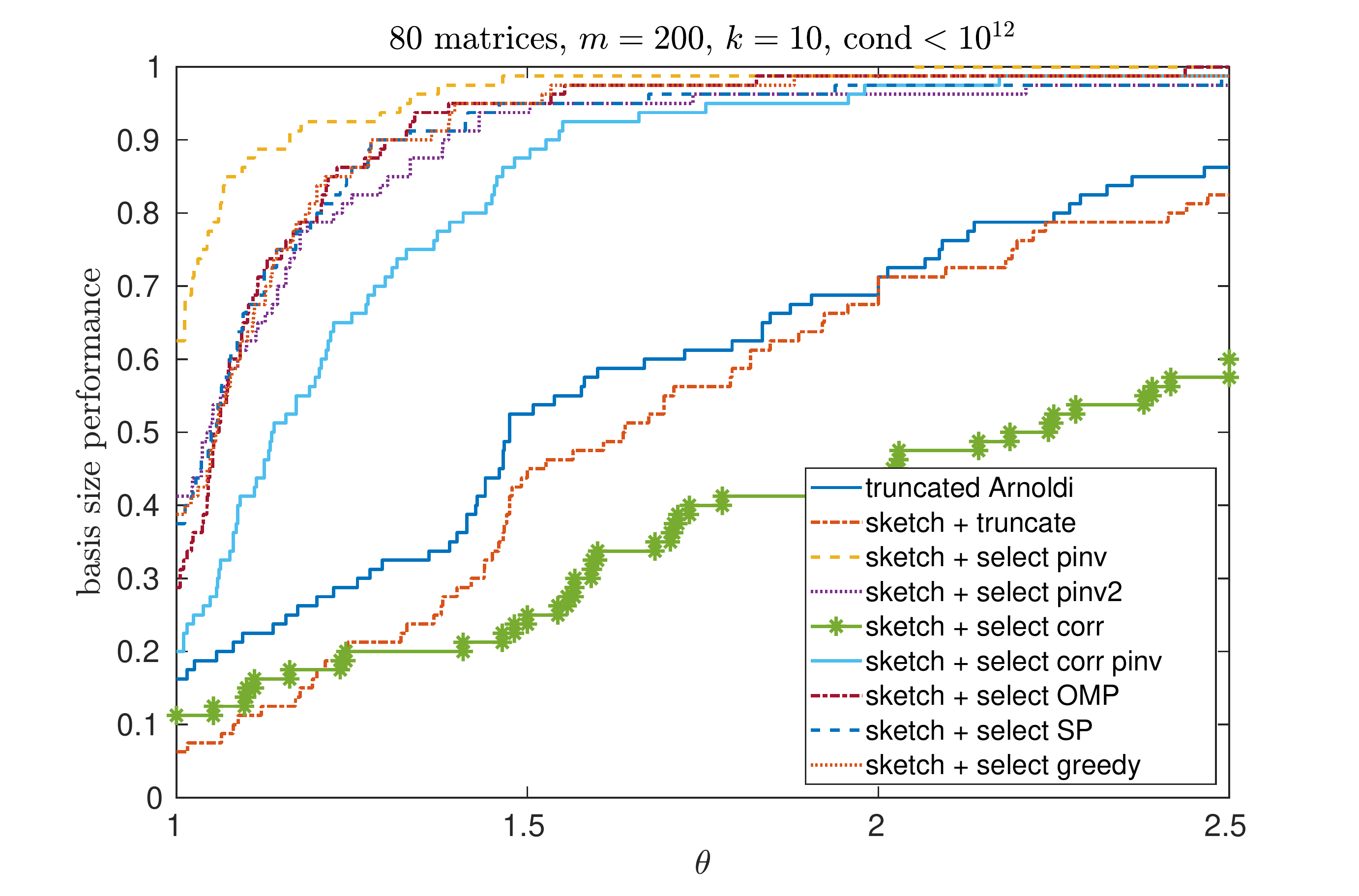}
    \bigskip%
    
    \includegraphics[scale=0.4]{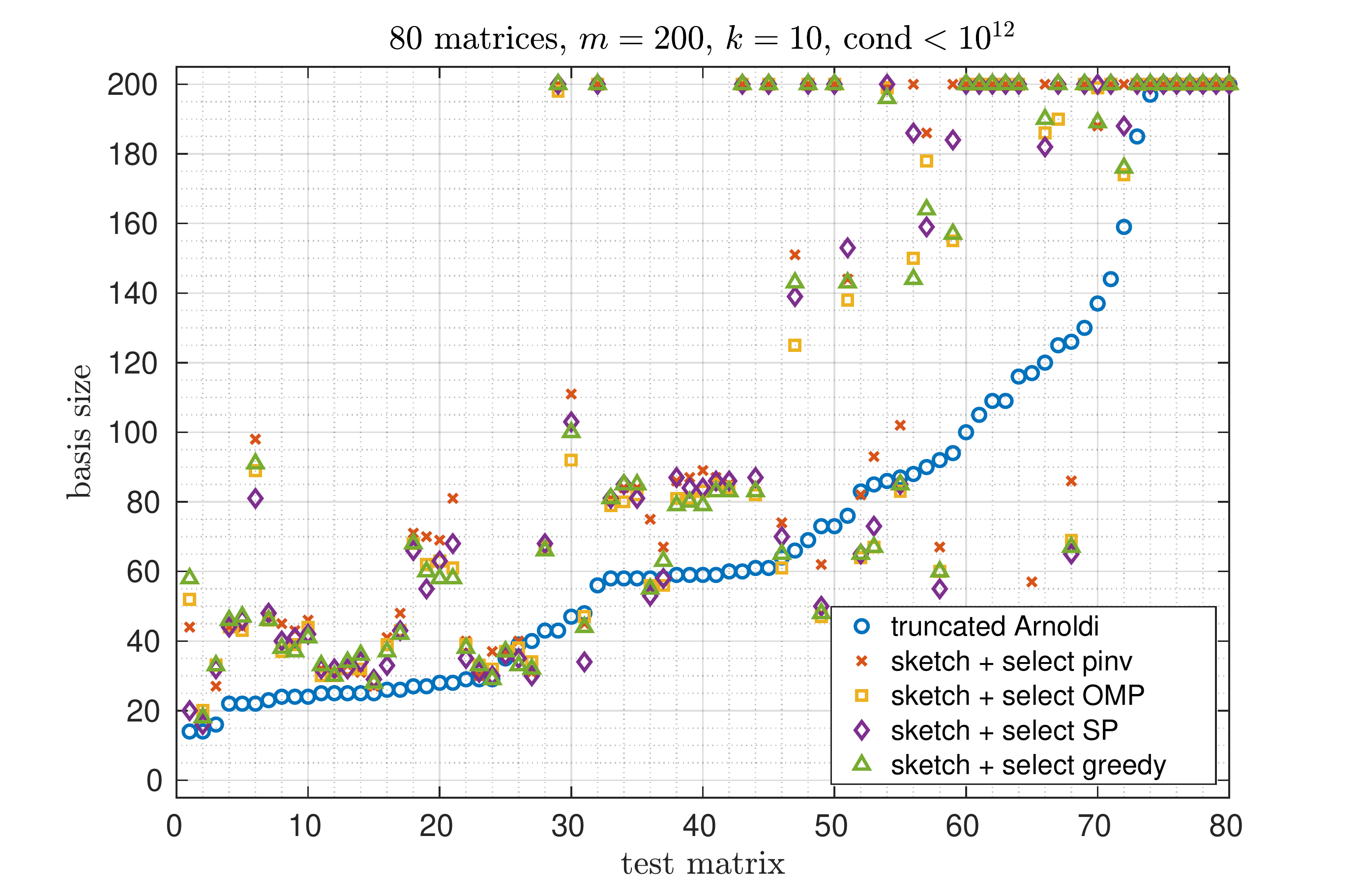}
    \caption{Basis sizes reached by the different methods, using a truncation parameter of $k=10$, with maximum basis size $m = 200$ and condition number bounded by $10^{12}$. Performance profiles (top) and basis sizes for the best performing methods (bottom).}
    \label{fig:test2c}
\end{figure}

\subsection{The effect of the starting vector}
\label{subsec:experiments--starting-vector}

In the previous experiments, we have used a starting vector $\vec b$ with random unit normally distributed entries. To investigate the influence of the starting vector on the performance of the algorithms, we repeat the experiment of \cref{fig:test2b} using as vector $\vb$ the first canonical unit vector $\ve_1$, instead of a random vector. The results are reported in \cref{fig:test3a}. \revone{Surprisingly, truncated Arnoldi performs significantly better with this starting vector, while the performance of the sketch-and-select variants degrades (relative to truncated Arnoldi). For example, the number of matrices for which a basis dimension of 100 is reached with truncated Arnoldi increases from~17 to~28 (comparing \cref{fig:test2b} with \cref{fig:test3a}).}

In \cref{fig:test3b} we repeat the same experiment by slightly perturbing the starting vector, i.e., we take $\vb = \ve_1 + 10^{-15} \ve$, where $\ve$ denotes the vector of all ones. With this change, the performance of the sketch-and-select Arnoldi variants improves significantly relative to truncated Arnoldi, though not to the same level of what was observed in the experiment in \cref{fig:test2b} with a random starting vector. A very similar improvement is obtained with a small random perturbation of the vector~$\vb$. 

While it currently appears to be impossible to make any general statements about the dependence of relative performance of sketch-and-select Arnoldi on the starting vector~$\vb$, we observed that truncated Arnoldi can produce artificially well conditioned bases for certain starting vectors. For example, it may happen that sparse basis vectors constructed by truncated Arnoldi have disjoint supports, and so they are all orthogonal to each other. (One example is the matrix \texttt{Norris/lung2}: when $\vb=\ve_1$, the first 452~Krylov basis vectors produced by truncated Arnoldi with truncation parameter $k\geq 2$ are given by $\pm\ve_{2j-1}$.) The sketching-based methods do not ``see'' the sparsity of the basis vectors and hence cannot produce this exact orthogonality, losing performance relative to truncated Arnoldi. Adding a small (random) perturbation to the starting vector~$\vb$ removes the sparsity and hence reduces the appearance of such artificial cases.

\begin{figure}[h!]
    \centering
    \includegraphics[scale=0.4]{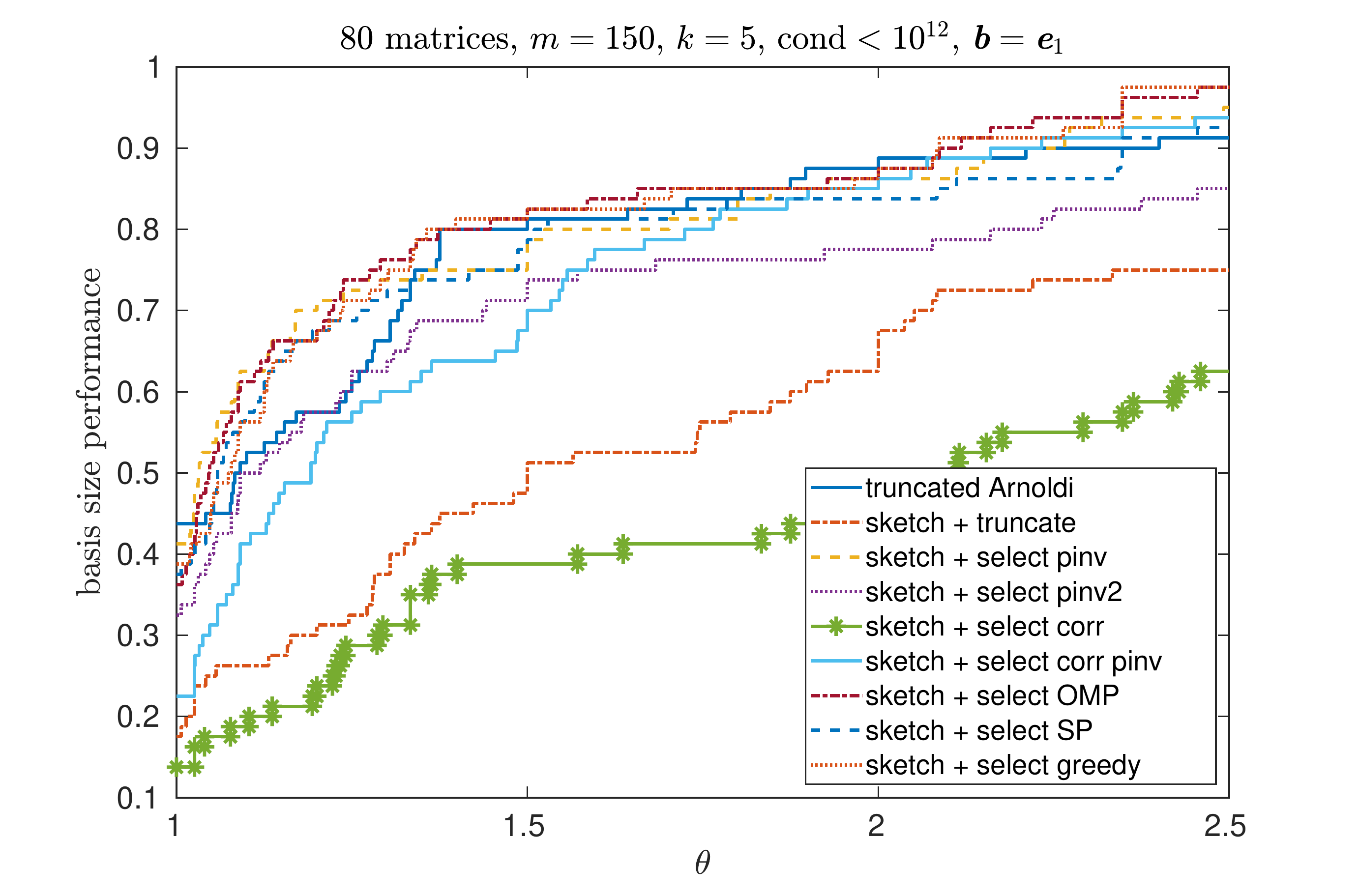}
    \bigskip%
    
    \includegraphics[scale=0.4]{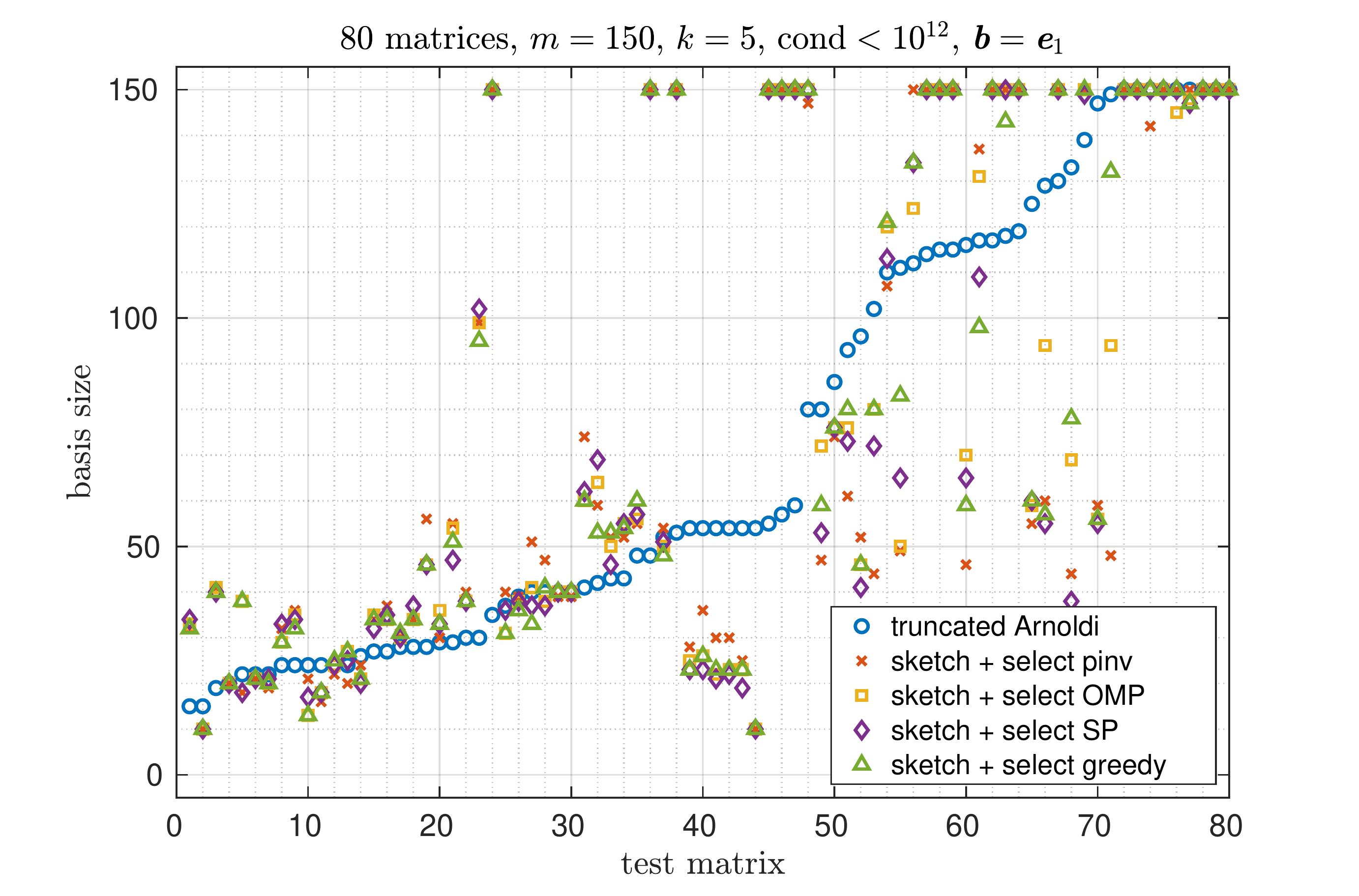}
    \caption{Basis sizes reached by the different methods, using a truncation parameter of $k=5$, with maximum basis size $m = 150$ and condition number bounded by $10^{12}$. The starting vector is $\vb = \ve_1$. Performance profiles (top) and basis sizes for the best performing methods (bottom).}
    \label{fig:test3a}
\end{figure}

\begin{figure}[h!]
    \centering
    \includegraphics[scale=0.4]{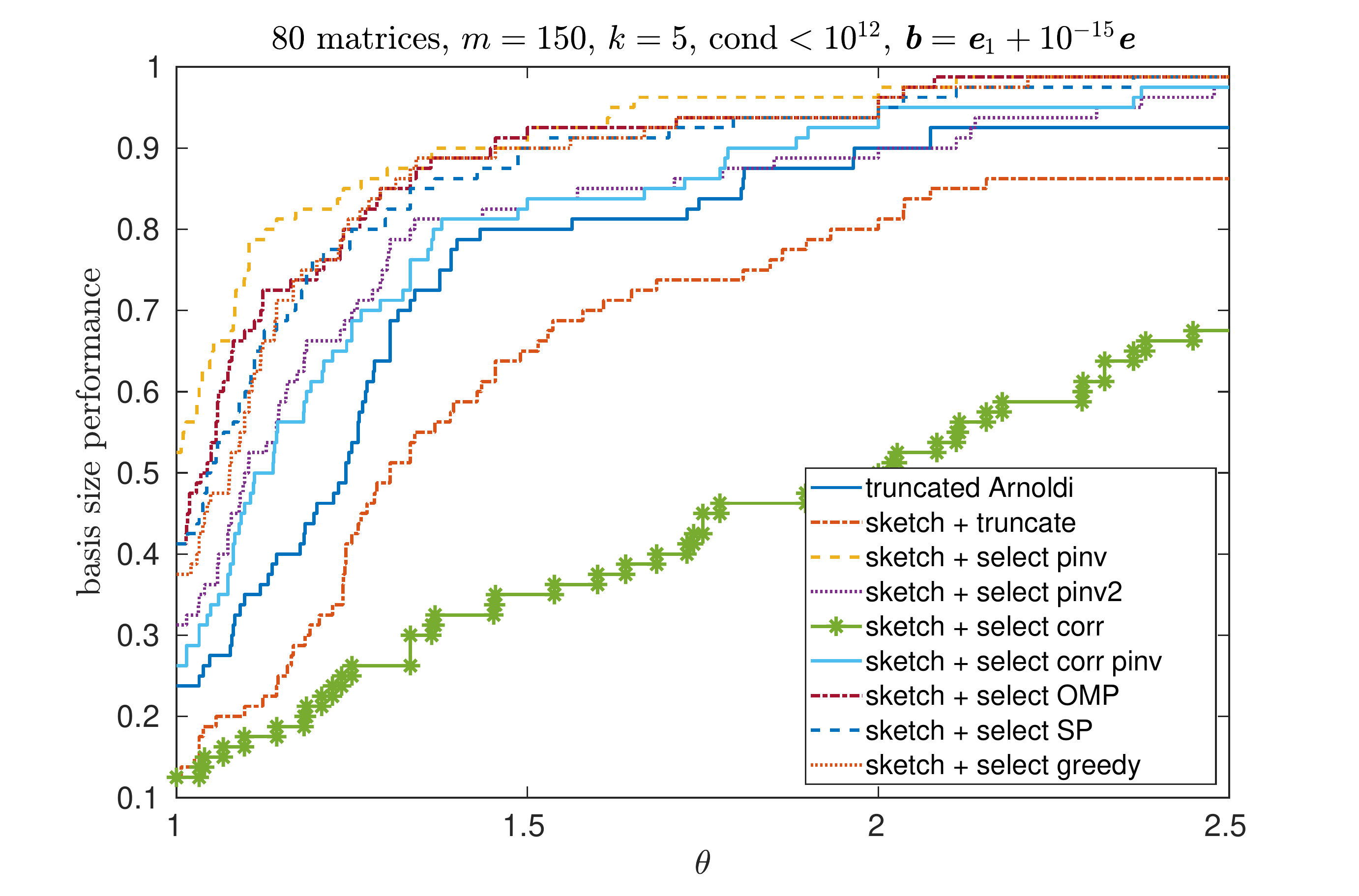}
    \bigskip%
    
    \includegraphics[scale=0.4]{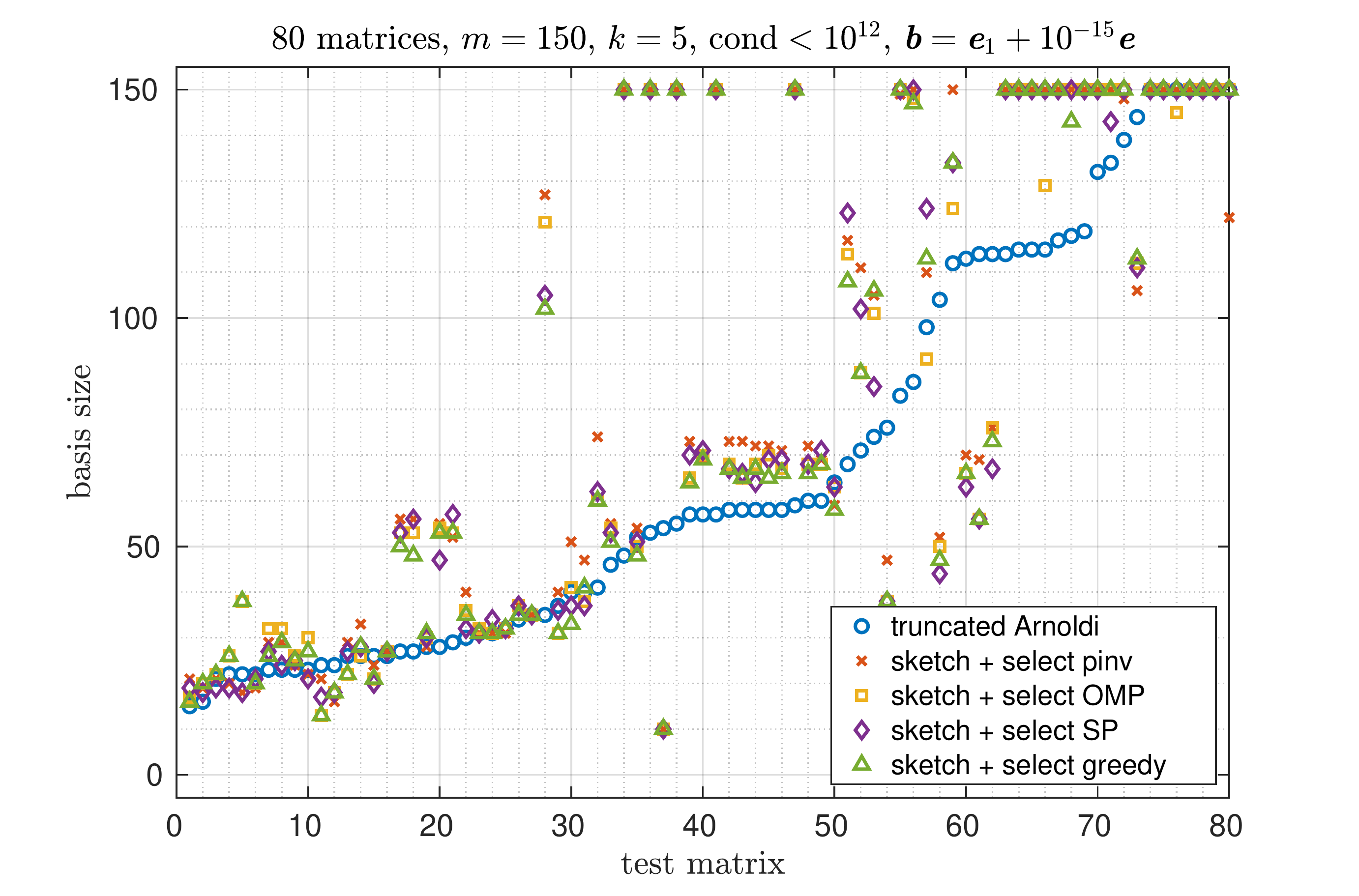}
    \caption{Basis sizes reached by the different methods, using a truncation parameter of $k=10$, with maximum basis size $m = 150$ and condition number bounded by $10^{12}$. The starting vector is $\vb = \ve_1 + 10^{-15}\ve$. Performance profiles (top) and basis sizes for the best performing methods (bottom).}
    \label{fig:test3b}
\end{figure}

\revone{

\subsection{Application to sketched GMRES}
\label{subsec:experiments-sgmres}

We now test the performance of the sketch-and-select Arnoldi variants as basis constructors within the sGMRES method for solving linear systems. 

For the experiments in this section, from the 80 test matrices used in the previous experiments we select the $10$ matrices on which GMRES is able to obtain a residual smaller than $10^{-3}$ within its first $200$ iterations.
For each matrix, we consider a linear system with a random unit norm right-hand side $\vb$ and we run GMRES (with an orthogonal Krylov basis generated by the non-truncated Arnoldi method) and sGMRES with truncated Arnoldi and the variants of sketch-and-select Arnoldi that performed best in the previous experiments, all with truncation parameter $k = 5$. The iterations are stopped once the residual is smaller than $10^{-8}$, or, in the case of sGMRES, when the condition number of the generated Krylov basis becomes larger than $10^{15}$. 
We visualize our results with a performance profile in which the horizontal axis displays the number $\theta$ of digits of accuracy lost in the residual with respect to the best method. For example, a value $\theta = 2$ means that the residual of a specific method is within a factor $100$ of the residual of the best method, which is always GMRES.

The results are shown in \cref{fig:test_sgmres}. We  observe that \texttt{ssa-pinv}, \texttt{ssa-OMP}, \texttt{ssa-SP} and \texttt{ssa-greedy} obtain on average two more digits of accuracy compared to truncated Arnoldi. All sketch-and-select variants perform similarly, with \texttt{ssa-pinv} being slightly better than the others.

\begin{figure}[h!]
    \centering
    \includegraphics[scale=0.4]{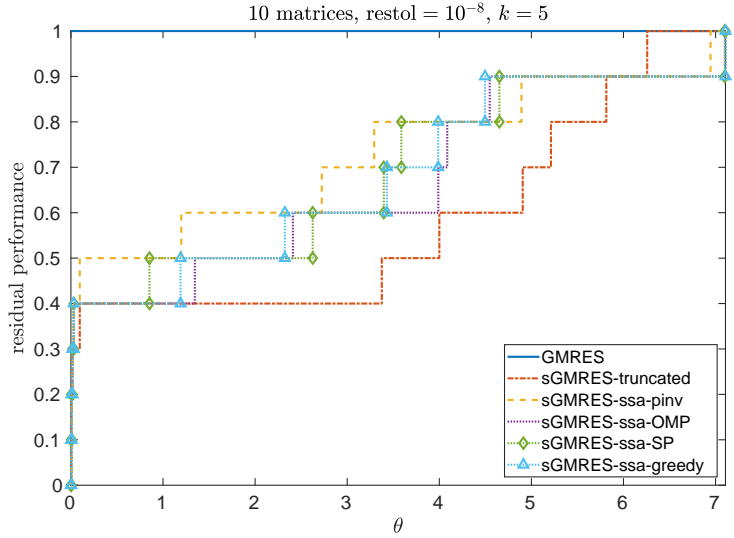}
    \caption{Performance profile of GMRES and sGMRES with different basis constructors, using a residual tolerance of $10^{-8}$, a truncation parameter of $k=5$ for sGMRES and a maximum basis condition number of $10^{15}$. \revtwo{The sketch-and-select methods obtain on average two more digits of accuracy compared to truncated Arnoldi}.}
    \label{fig:test_sgmres}
\end{figure}

\subsection{sGMRES with very ill-conditioned bases}\label{sec:ill}

While running the experiments, we noticed that for several test problems sGMRES continued to converge even after the condition number of the generated Krylov basis became larger than $10^{15}$. This phenomenon does not occur consistently and we were unable to find a clear theoretical explanation. Hence, we preferred not to rely on this behavior and rather stop when the basis becomes ill-conditioned.
The sGMRES method often continues to converge both when the basis is constructed with truncated Arnoldi, and with \texttt{ssa-OMP}, \texttt{ssa-SP} or \texttt{ssa-greedy}, but we did not observe it when using \texttt{ssa-pinv}. We were able to partially explain this behavior by looking at the singular value distribution of the Krylov basis constructed by the different methods. 

\begin{figure}[h!]
    \centering
    \includegraphics[scale=0.39]{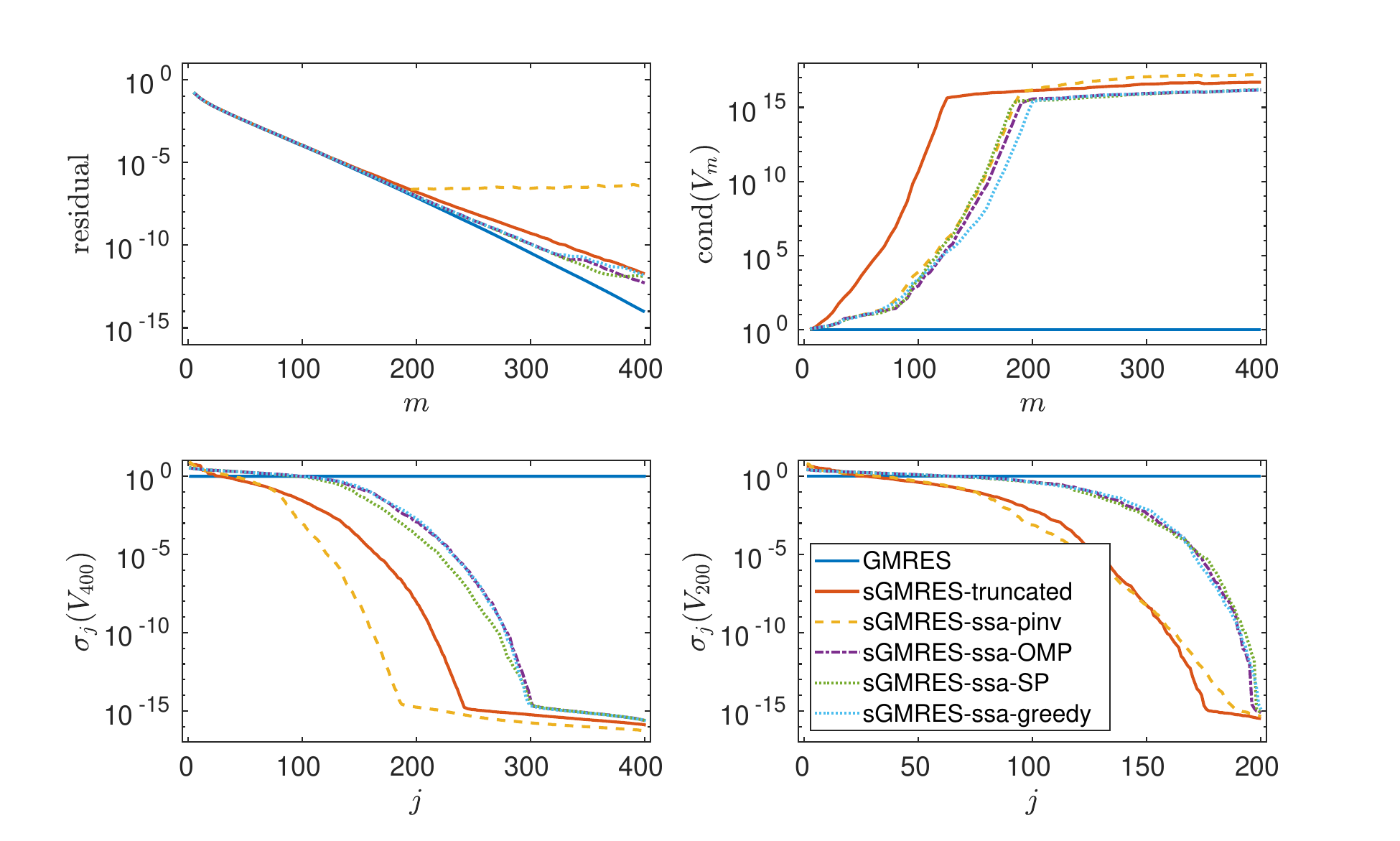}
    \caption{\revone{Convergence of GMRES and sGMRES (top left), basis condition number growth (top right) and singular values of $V_m$ for basis dimension \revtwo{$m = 400$} and \revtwo{$m = 200$} (bottom left and right, respectively), using different basis constructors, for the test problem \texttt{Norris/torso3}. All methods, except \texttt{sGMRES-ssa-pinv}, continue to converge even when $\cond(V_m)>10^{15}$.}}
    \label{fig:test_sgmres_singval}
\end{figure}

In \cref{fig:test_sgmres_singval} we show the convergence behavior of sGMRES for the test problem \texttt{Norris/torso3} (size $N = 259156$), as well as the growth of the basis condition number, and the distribution of singular values for the bases of size \revtwo{$m = 200$} and \revtwo{$m = 400$} generated with the different methods. Although the smallest singular values of the bases constructed using the sketch-and-select Arnoldi variants have roughly the same order of magnitude, we  see that the singular value distribution for the \texttt{ssa-pinv} basis is significantly different, with the number of small singular values being quite large relative to the other sketch-and-select methods. This is most evident for the basis of dimension~\revtwo{$m = 200$}. This difference implies that the basis constructed with \texttt{ssa-pinv} has a smaller numerical rank compared to the other bases, and it is likely that this is causing sGMRES to stagnate. On the other hand, the basis constructed by truncated Arnoldi has a condition number of order $10^{15}$ already for \revtwo{$m = 100$}, and for \revtwo{$m = 200$} its singular value distribution is quite similar to the one obtained with \texttt{ssa-pinv}. Nevertheless, sGMRES with truncated Arnoldi continues to converge and suffers only a small delay. 

\begin{figure}[h!]
    \centering
    \includegraphics[scale=0.39]{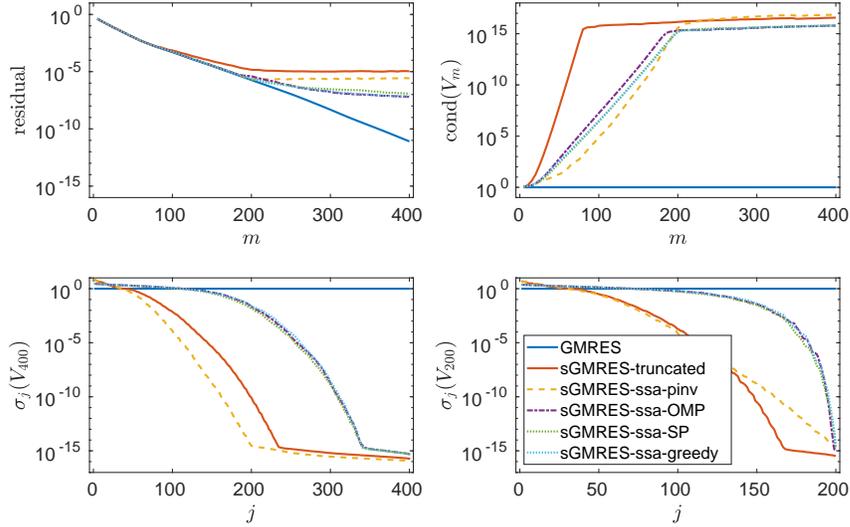}
    \caption{\revone{Convergence of GMRES and sGMRES (top left), basis condition number growth (top right) and singular values of $V_m$ for basis dimension \revtwo{$m = 400$} and \revtwo{$m = 200$} (bottom left and right, respectively), using different basis constructors, for the test problem \texttt{Norris/stomach}. All methods display delays in convergence after $\cond(V_m) > 10^{15}$.}}
    \label{fig:test_sgmres_singval2}
\end{figure}

In \cref{fig:test_sgmres_singval2}, we show the results of the same experiment for the test problem \texttt{Norris/stomach} (size $N = 213360$), where the behavior of sGMRES is significantly different. Even though the condition number growth and singular value distributions of the generated Krylov bases are very similar to the ones in \cref{fig:test_sgmres_singval}, for all basis generation processes sGMRES starts to stagnate or exhibits a significant slowdown in convergence shortly after the condition number of the Krylov basis becomes larger than $10^{15}$.

\revtwo{
\begin{figure}[h!]
    \centering
    \includegraphics[scale=0.39]{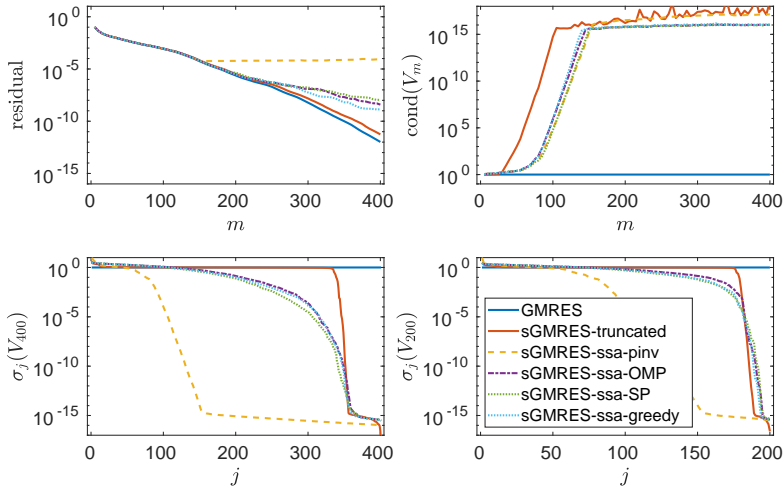}
    \caption{\revtwo{Convergence of GMRES and sGMRES (top left), basis condition number growth (top right) and singular values of $V_m$ for basis dimension \revtwo{$m = 400$} and \revtwo{$m = 200$} (bottom left and right, respectively), using different basis constructors, for the test problem \texttt{FEMLAB/poisson3Db}. The basis constructed by truncated Arnoldi is the most ill-conditioned, but it has fewer small singular values.}}
    \label{fig:test_sgmres_singval3}
\end{figure}
}

\revtwo{
In \cref{fig:test_sgmres_singval3}, we show the results of the same experiment for the test problem \texttt{FEMLAB/poisson3Db} (size $N = 85623$), where sGMRES with truncated Arnoldi performs better than with sketch-and-select. In this case, even if the condition number of the basis generated with truncated Arnoldi grows more quickly than with the sketch-and-select methods, by looking at the distribution of singular values we see that there is only a small number of very small singular values, while most of the singular values are close to one. On the other hand, \texttt{ssa-pinv} constructs a basis with slightly slower condition number growth, but it has a significantly higher number of small singular values, especially for large $m$. This example highlights the fact that the singular value distribution of the Krylov basis can play an important role in the convergence behavior of sGMRES.
}

The examples in \cref{fig:test_sgmres_singval,fig:test_sgmres_singval2,fig:test_sgmres_singval3} demonstrate that the convergence behavior of sGMRES with an ill-conditioned basis remains an open problem. Further analysis of these situations will be needed to fully understand the behavior of sGMRES.
}

\revtwo{
\begin{figure}[h!]
    \centering
    \includegraphics[scale=0.4]{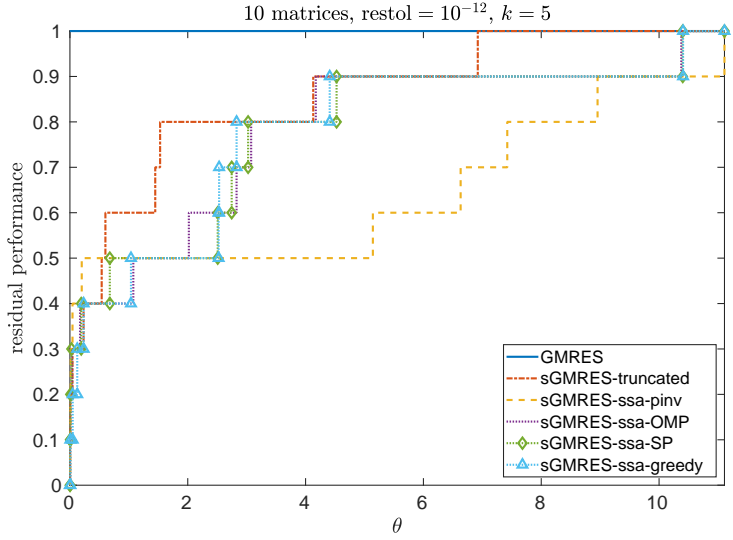}
    \caption{\revtwo{Performance profile of GMRES and sGMRES with different basis constructors, using a residual tolerance of $10^{-12}$ and a truncation parameter of $k=5$ for sGMRES. On each problem, all methods are run for the same number of iterations, regardless of basis condition number growth.}}
    \label{fig:test_sgmres_fixiter}
\end{figure}
}

\revtwo{In light of these observations, we have also tested the performance of sGMRES by allowing the method to proceed even after the basis becomes ill-conditioned. For each of the $10$ test problems, we run GMRES until its residual is smaller than $10^{-12}$, and we run all variants of sGMRES for the same number of iterations as GMRES, without monitoring the basis condition number. The results are displayed in \cref{fig:test_sgmres_fixiter} with a performance profile, where again~$\theta$ denotes the number of digits of accuracy lost with respect to the best method. 
In this case sGMRES with truncated Arnoldi performs better than with sketch-and-select Arnoldi, and on several problems it is able to gain one or two digits of accuracy with respect to the best sketch-and-select methods. Note that sGMRES with \texttt{ssa-pinv} performs particularly poorly, because it consistently stops converging when the basis becomes ill-conditioned, in contrast with the other methods that often keep converging. It appears that sGMRES with truncated Arnoldi benefits from this phenomenon more often than the sketch-and-select methods.

}

\section{Growth of the basis condition number}\label{sec:cond}

\revtwo{
In this section, we establish bounds on how the condition number of a basis $V_j$ can change upon the addition of the next basis vector $\vv_{j+1}$. This analysis applies to any basis construction procedure and it is not restricted to sketch-and-select Arnoldi. In addition to proving upper bounds on the condition number growth, we also identify some fundamental limitations of non-orthogonal basis construction approaches, including truncated Arnoldi and the sketch-and-select Arnoldi procedure that we have introduced.
}

By \cref{eq:basiscond}, \revtwo{to bound the growth of $\cond([V_j, \vv_{j+1}])$} it is sufficient to  only control the condition number growth of the sketched Krylov basis $[SV_j, S \vv_{j+1}]$. For notational convenience, we will write this as $[V,\vv]$. Our aim is to bound the growth of $\cond([V , \vv])$ in terms of the current $\cond(V)$. Note that the Gram matrix $[V, \vv]^T [V, \vv]$ has unit diagonal in the sketch-and-select Arnoldi process as we always normalize each sketched Krylov basis vector, and so 
\[
 \cond(V)^2 = \cond\left( \begin{bmatrix} V^T V & 0 \\ 0^T & 1 \end{bmatrix} \right).
\]
Also, 
\[
 \cond([V,\vv])^2 = \cond\left( \begin{bmatrix} V^T V & V^T \vv  \\ \vv^T V & 1 \end{bmatrix} \right).
\]
We can now apply standard relative perturbation bounds known for symmetric positive definite matrices. To this end, write 
\[
\begin{bmatrix} V^T V & V^T \vv  \\ \vv^T V & 1 \end{bmatrix} 
= 
\begin{bmatrix} V^T V & 0  \\ 0^T & 1 \end{bmatrix} 
+
\begin{bmatrix} O & V^T \vv  \\ \vv^T V & 0 \end{bmatrix} 
=: G + \Delta G.
\]
Note that $\|\Delta G\|=\|V^T \vv\|$. Then (see, e.g., \cite{demmel1992jacobi} or \cite[Thm.~2.3]{Mathias})
\begin{eqnarray*}
\cond([V,\vv])^2 &=& \cond(G + \Delta G) = \frac{\lambda_{\max}(G + \Delta G)}{\lambda_{\min}(G + \Delta G)} \\
&\leq& 
\frac{(1+\eta) \lambda_{\max}(G)}{(1-\eta)\lambda_{\min}(G)} 
= 
\frac{1+\eta}{1-\eta} \cond(V)^2, 
\end{eqnarray*}
where 
\[
\eta = \|G^{-1/2} (\Delta G) G^{-1/2} \| = \| (V^T V)^{-1/2} V^T \vv \| \leq \sigma_{\min}(V)^{-1} \| V^T \vv \|.
\]
\smallskip
Clearly this bound is only useful as long as $1-\eta>0$, which is guaranteed if $\|V^T \vv\| < \sigma_{\min}(V)$, and generally it cannot be expected to be sharp. However, it shows that it is a good idea to try to keep $\|V^T \vv\|$ as small as possible. 
Going back to the notation used for the sketch-and-select Arnoldi process, this means that we should aim to keep $\| (SV_j)^T (S \vv_{j+1}) \|$ small.

\smallskip

\revtwo{In the following, we use} an alternative approach to quantify the condition number growth, by bounding the smallest and largest eigenvalue of the Gram matrix, taking into account the special structure of that matrix. \revtwo{In addition to bounding the condition number growth, this more constructive approach allows us to identify a vector~$\vv$ that approximately attains the upper bound on the condition number.} The following theorem provides such a result, giving a more explicit bound on \revone{$\sigma_{\min}([V, \vv])$ and, as a consequence, on $\cond([V, \vv])$}.

\begin{theorem}
\label{thm:basis-condition-number-bound}
Let $V$ be a matrix with $m$ linearly independent columns of unit norm. Denote by $\sigma_{\min}$ the smallest singular value of $V$ and let $\vv$ be a unit norm vector such that $\|V^T \vv\| < \sigma_{\min}$. Then
\revone{
\begin{equation}
    \label{eq:sigmamin-thm-bound}
    \sigma_{\min}([V, \vv])^2 \ge \frac{1 + \sigma_{\min}^2 -  \sqrt{\strut (1-\sigma_{\min}^2)^2 + 4\|V^T \vv \|^2}}{2},
\end{equation}
and there exists a vector $\vv$ that attains the lower bound \cref{eq:sigmamin-thm-bound}. For the matrix $[V, \vv]$ constructed with the vector $\vv$ that attains \cref{eq:sigmamin-thm-bound}, we have 
\begin{equation}
    \label{eq:condition-number-thm-bound--attainable}
    \cond([V, \vv])^2 \ge \frac{2}{1 + \sigma_{\min}^2 -  \sqrt{\strut (1-\sigma_{\min}^2)^2 + 4\|V^T \vv \|^2}}.
\end{equation}
}
\end{theorem}

\begin{proof}
We begin by remarking that, as the columns of $[V,\vv]$ are normalized, we have  $\sigma_{\max}([V,\vv])\leq \sqrt{m+1}$ and hence  any  ill-conditioning of $[V,\vv]$ is mainly attributable to a small value of $$\sigma_{\min}([V,\vv]) = \lambda_{\min}([V,\vv]^T [V,\vv])^{1/2}.$$ 

Define the Rayleigh quotient
\[
R(\vx,\vv) = \left[ \vx^T, \sqrt{1-\|\vx\|^2}\right] 
\begin{bmatrix} V^T V & V^T \vv  \\ \vv^T V & 1 \end{bmatrix} 
\begin{bmatrix} \vx \\  \sqrt{1-\|\vx\|^2} \end{bmatrix}, \quad  \|\vx\|\leq 1.
\]
Let us denote $\beta = \|\vx\|$ and fix $\| V^T \vv\|:=\alpha$. Then
\begin{equation}
\label{eq:ray}
R(\vx,\vv) = \vx^T (V^T V)\vx + 2 \vx^T (V^T \vv) \sqrt{1-\beta^2} + 1 - \beta^2.
\end{equation}
The first term in this Rayleigh quotient is minimized  by choosing $\vx$ as an eigenvector of $V^T V$ corresponding to $\lambda_{\min} = \lambda_{\min}(V^T V)$. For any choice of $\vx$, $\|\vx\|=\beta$, the second term is minimal when $\vv$ is such that $V^T \vv = -\alpha \vx/\beta$. Hence, we can minimize the overall Rayleigh quotient directly, leading to 
\[
    R_{\min}(\beta) = \beta^2 \lambda_{\min} - 2 \alpha \beta \sqrt{1-\beta^2} + 1 - \beta^2.
\]
To find the optimal $\beta\in [0,1]$, we set $\gamma:= 1-\beta^2$, upon which $R_{\min}(\beta) = (1-\gamma)\lambda_{\min} - 2\alpha \sqrt{\gamma-\gamma^2} +\gamma$ which is easy to differentiate for $\gamma$. 
The optimal value $\beta_*$ that minimizes $R_{\min}$ is
\[
    \beta_* = \sqrt{1-\gamma_*}, \quad \text{where} \quad  \gamma_* = \frac{C_*^2 - C_*\sqrt{C_*^2+4}  + 4}{2(C_*^2+4)}, \quad C_* = \frac{1-\lambda_{\min}}{\alpha}.
\]

We can now derive a rather simple expression for $R_\text{min}(\beta_*)$ in terms of $C_*$. We have
\[
    \gamma_* = \frac{C_*^2 - C_*\sqrt{C_*^2+4}  + 4}{2(C_*^2+4)} = \frac{1}{2} - \frac{1}{2} \frac{C_*}{\sqrt{C_*^2 + 4}},
\]
and from this expression it easily follows that
\[
\sqrt{\gamma_* - \gamma_*^2} = \frac{1}{\sqrt{C_*^2 + 4}}.
\]
By plugging this expression and $\gamma_*$ into the expression of $R_\text{min}(\beta_*)$, we have
\begin{align*}
    R_\text{min}(\beta_*) &= \lambda_\text{min} + (1 - \lambda_\text{min})\gamma_* - 2 \alpha \sqrt{\gamma_* - \gamma_*^2} \\ 
    &= \lambda_\text{min} + \alpha C_* \Big( \frac{1}{2} - \frac{1}{2} \frac{C_*}{\sqrt{C_*^2 + 4}} \Big) - \frac{2 \alpha}{\sqrt{C_*^2 + 4}} \\
    &= \lambda_{\text{min}} + \frac{\alpha}{2 \sqrt{C_*^2 + 4}} \Big( C_* \sqrt{C_*^2 + 4} - C_*^2 - 4 \Big) \\
    &= \lambda_{\text{min}} + \frac{\alpha}{2} \Big( C_* - \sqrt{C_*^2 + 4} \Big) \\
    &= \frac{1}{2} + \frac{1}{2}\lambda_\text{min} - \frac{\alpha}{2} \sqrt{C_*^2 + 4} \\[2mm]
    &= \frac{1 + \lambda_\text{min} -  \sqrt{(1-\lambda_\text{min})^2 + 4\alpha^2}}{2}.
\end{align*}
For this quantity to be positive, we require $\alpha^2 < \lambda_{\min}$ or equivalently $\|V^T \vv\| < \sigma_{\min}(V)$.

\revone{This shows that $\sigma_\text{min}([V, \vv])^2 \ge R_\text{min}(\beta_*)$ and this lower bound is attained by $\vv$ such that $V^T \vv = -\alpha \vx / \beta_*$, where $\vx$ is the eigenvector of $V^T V$ associated to $\lambda_\text{min}(V^T V)$.
This proves the first part of the theorem, recalling that $\lambda_{\min} = \sigma_{\min}(V)^2$ and $\alpha = \|V^T \vv\|$.}
\revone{The second part follows from \cref{eq:sigmamin-thm-bound}, by using $\cond([V, \vv]) = \sigma_{\max}([V, \vv])/\sigma_{\min}([V, \vv])$ and the fact that $\sigma_{\max}([V, \vv]) \ge 1$.}
\end{proof}

\revone{
The following corollary is a simple consequence of \cref{thm:basis-condition-number-bound}, \revtwo{and provides an upper bound on the basis condition number growth}.
\begin{corollary}
\label{cor:basis-condition-number-bound}
Let $V$ be a matrix with $m$ linearly independent columns of unit norm. Denote by $\sigma_{\min}$ and $\sigma_{\max}$ the smallest and largest singular value of $V$, respectively. Further, let $\vv$ be a unit norm vector such that $\|V^T \vv\| < \sigma_{\min}$. Then
\begin{equation}
    \label{eq:condition-number-thm-bound}
    \cond([V,\vv])^2 \leq \frac{1 + {\sigma_{\max}^2} +  \sqrt{\strut (\sigma_{\max}^2-1)^2 + 4\|V^T \vv\|^2}}{1 + \sigma_{\min}^2 -  \sqrt{\strut (\sigma_{\min}^2-1)^2 + 4\|V^T \vv\|^2}}.
\end{equation}
\end{corollary}
}

\begin{proof}
\revone{We already have the bound \cref{eq:sigmamin-thm-bound} from \cref{thm:basis-condition-number-bound} for $\sigma_{\min}([V, \vv])$.}
\revone{With the same approach used in the proof of \cref{thm:basis-condition-number-bound}}, the first term in the Rayleigh quotient \cref{eq:ray} can be \emph{maximized} by choosing~$\vx$ as an eigenvector of $V^T V$ corresponding to $\lambda_{\max} = \lambda_{\max}(V^T V)$. For any choice of~$\vx$, $\|\vx\|=\beta$, the second term in \cref{eq:ray} is  maximal when $\vv$ is such that $V^T \vv = \alpha \vx/\beta$. Hence, we can maximize the overall Rayleigh quotient directly, leading to 
\[
    R_{\max}(\beta) = \beta^2 \lambda_{\max} + 2 \alpha \beta \sqrt{1-\beta^2} + 1 - \beta^2.
\]
To find the optimal $\beta\in [0,1]$, we set $\gamma:= 1-\beta^2$, upon which $R_{\max}(\beta) = (1-\gamma)\lambda_{\max} + 2\alpha \sqrt{\gamma-\gamma^2} +\gamma$. 
 The optimal value $\beta^*$ that maximizes $R_{\max}$ is
\[
    \beta^* = \sqrt{1-\gamma^*}, \quad \text{where} \quad  \gamma^* = \frac{C^2 - C \sqrt{C^2+4}  + 4}{2(C^2+4)}, \quad C = \frac{\lambda_{\max}-1}{\alpha}.
\]
(Note that the only difference with respect to the expression for $\beta_*$ is in $C$ versus $C_*$.)
Evaluating $R_{\max}(\beta^*)$  yields 
\[
 R_\text{max}(\beta^*)  = 
 \frac{1 + \lambda_\text{max} +  \sqrt{(\lambda_\text{max}-1)^2 + 4\alpha^2}}{2}.
\]
Combining the expressions for the (worst-case) Rayleigh quotients, we obtain 
\[
\cond([V,\vv])^2 \leq \frac{R_{\max}(\beta^*)}{R_{\min}(\beta_*)} 
= \frac{1 + \lambda_\text{max} +  \sqrt{(\lambda_\text{max}-1)^2 + 4\alpha^2}}{1 + \lambda_\text{min} -  \sqrt{(\lambda_\text{min}-1)^2 + 4\alpha^2}}.
\]
The result follows since $\lambda_{\max} = \sigma_{\max}(V)^2$, $\lambda_{\min} = \sigma_{\min}(V)^2$, and $\alpha = \|V^T \vv\|$.
\end{proof}

\revone{Note that the \revtwo{upper} bound \cref{eq:condition-number-thm-bound} on $\cond([V, \vv])$ is not necessarily attainable, since we have used two different vectors $\vv$ to respectively maximize $R_{\max}(\beta^*)$ and minimize $R_{\min}(\beta_*)$.}
\revone{On the other hand, \cref{thm:basis-condition-number-bound} states that there exists a choice of $\vv$ for which $\cond([V, \vv])$ satisfies the \revtwo{lower} bound \cref{eq:condition-number-thm-bound--attainable}.}

Recalling that $\sigma_\text{max}([V, \vv]) \le \sqrt{m+1}$, we find that the right-hand side of \cref{eq:condition-number-thm-bound--attainable} is smaller than the right-hand side of \cref{eq:condition-number-thm-bound} at most by a factor $m+1$.
In principle, it is possible to select a vector $\vv$ that realizes \cref{eq:condition-number-thm-bound--attainable} at every iteration, even though this is unlikely to happen in practice. 

\revtwo{In the discussion below we show the following: if at every iteration of a non-orthogonal basis construction procedure we add to the basis a vector that attains the lower bound in~\cref{eq:sigmamin-thm-bound}, the condition number of the basis can grow exponentially even if the assumption $\norm{V^T \vv} < \sigma_\text{min}(V)$ is always satisfied.}

Going back to the notation used for the sketch-and-select Arnoldi process, let us consider the behavior of $\sigma_\text{min}(SV_m)$ as $m$ increases, assuming that at each iteration \revtwo{we select}~$S \vv_{m+1}$ as a vector $\vv$ that satisfies \cref{eq:sigmamin-thm-bound} \revtwo{with an equality}. 
For convenience, define $x_m := \sigma_\text{min}(S V_m)$ and $\alpha_m := \|S \vv_{m+1}\| < x_m$.
Because of \cref{eq:sigmamin-thm-bound}, these quantities satisfy the recurrence relation
\begin{equation*}
    x_{m+1}^2 = \frac{1}{2} \left(1 + x_m^2 - \sqrt{(1 - x_m^2)^2 + 4 \alpha_m^2}\right), \qquad m \ge 1,
\end{equation*}
with $x_1 = \sigma_\text{min}(S\vv_1) = 1$.
Using the fact that $\sqrt{1 + z} \ge 1 + \frac{1}{2}z - \frac{1}{8}z^2$ for all $z \ge 0$, we can show that
\begin{equation*}
    x_{m+1}^2 \le x_m^2 - \frac{\alpha_m^2}{1 - x_m^2} \left(1 - \frac{\alpha_m^2}{(1 - x_m^2)^2}\right), \qquad \revtwo{m \ge 1}.
\end{equation*}
If for instance we \revtwo{assume that} $x_{\revtwo{m_0}} \le 1/\sqrt{2}$ and $\revone{\alpha_m = \frac{1}{2} x_m}$ \revtwo{for all $m \ge m_0$}, we have
\begin{equation*}
    1 - \frac{\alpha_m^2}{(1 - x_m^2)^2} \revone{\ge} \frac{1}{2}, \qquad \revtwo{m \ge m_0,}
\end{equation*}
and so we obtain
\begin{equation*}
    x_{m+1}^2 \le x_m^2 - \frac{1}{2} \alpha_m^2 \le \frac{7}{8} x_m^2, \qquad m \ge \revtwo{m_0}.
\end{equation*}
This implies that
\begin{equation}
\label{eqn:exponential-growth}
    \sigma_{\text{min}}(S V_m) \le \left(\frac{7}{8}\right)^{\revtwo{(m-m_0)}/2} \revtwo{\sigma_\text{min}(S V_{m_0})}, \qquad m \ge \revtwo{m_0},
\end{equation}
showing that \revone{there exists a sequence of vectors $\{\vv_{j+1}\}_{j = 1, \dots, m-1}$, satisfying the condition $\norm{(SV_j)^T S\vv_{j+1}} = \frac{1}{2} \sigma_{\text{min}}(SV_j)$ for all $j$, such that the smallest singular value of $SV_m$ converges geometrically to $0$}. As a consequence, \revtwo{with this choice} the condition number of $SV_m$ \revtwo{diverges} geometrically (and hence also $\cond(V_m)$, because of \cref{eq:basiscond}), even if we \revtwo{have imposed} the condition $\norm{(SV_m)^T S\vv_{m+1}} = \frac{1}{2} \sigma_{\text{min}}(SV_m)$ \revone{at each iteration}, which is quite stringent \revtwo{especially when $\sigma_\text{min}(SV_m)$ becomes very small}.

\revone{The above results on the geometric increase of the basis condition number put a theoretical limitation on any non-orthogonal Krylov basis generation approach, including the sketch-and-select Arnoldi methods we have introduced. This may be considered a gloomy outlook for sketching-based Krylov methods. However, one should keep in mind that the situation considered here is a \emph{worst-case scenario} in which a worst possible vector $\vv$ is chosen at each iteration. In practice, we find that the onset of exponential condition number increase can often be delayed with specific choices of the vectors~$\vv$, such as the choices made by our sketch-and-select variants.}

\section{Further remarks on the subset selection problem}

Selecting $k$ columns that give the smallest condition number of $[V,\vv]$ in the sketch-and-select Arnoldi process is a combinatorial problem, and even selecting a near-best index set $I$ is nontrivial.  We now give examples demonstrating that neither the largest coefficients in $V^{\dagger} \vw$ nor those in $V^T \vw$ do necessarily indicate the  best vectors to select for a minimal condition number growth. Note that most greedy algorithms, including the ``Algorithm Greedy'' in \cite{natarajan1995sparse}, OMP~\cite{pati1993orthogonal}, and SP~\cite{dai2009subspace} use the entries of either  $V^{\dagger} \vw$ or $V^T \vw$   to select vectors, and so can be misled on examples like the ones below.  
Consider 
\[
V = \frac{1}{\sqrt{5}}\begin{bmatrix}
1 & 0 & 0 \\
      2 & 2 & 0 \\
      0 & 1 &  1 \\
      0 & 0 & 2 
\end{bmatrix}, \quad
\vec{w} = \begin{bmatrix}
   8\\
   8\\
   9\\
   7
\end{bmatrix}.
\]
Note that all columns of $V$ have unit norm. Say $k=1$, then which of the three columns of $V$ should we project out of $\vw$ to get the best possible  conditioned basis of four vectors?
Let us compute
\[
    V^\dagger \vec{w} = \begin{bmatrix}
    9.39\\
    1.68\\
    9.95
\end{bmatrix}, \quad 
V^T \vw = \begin{bmatrix}
10.7\\
   11.2\\
   10.3
\end{bmatrix}.
\]
According to these coefficients, we might be tempted to project $\vw$ either against $\vec{v}_3$ or $\vec{v}_2$, respectively. 
However, among the vectors
\[
 \widehat{\vec{v}}_i := (I - \vec{v}_i\vec{v}_i^\dagger)  \vw, \quad i=1,2,3,
\]
it is the choice $i=1$ that strictly minimizes both $\cond([V, \widehat{\vec{v}}_i])$ and $\cond ([V, \widehat{\vec{v}}_i/{\|\widehat{\vec{v}}_i\|}])$.

Instead of condition number growth, perhaps a better measure to look at is the growth of loss of orthogonality, e.g., by any of the metrics
\[
    \left\| I - [V,\vec{v}]^T[V,\vec{v}] \right\|,  \quad 
    \left\| I - \left[V,\frac{\vec{v}}{\|\vec{v}\|}\right]^T\left[V,\frac{\vec{v}}{\|\vec{v}\|}\right] \right\|, 
\]
\[
     \left\| [V,\vec{v}]^T[V,\vec{v}] \right\|,  \quad 
    \left\| \left[V,\frac{\vec{v}}{\|\vec{v}\|}\right]^T\left[V,\frac{\vec{v}}{\|\vec{v}\|}\right] \right\|
\]
in the Euclidean or Frobenius norm. Can we get some guarantees for that? It turns out that this  is also not the case. The above matrix $V$ and the vector $\vw=[9,9,10,10]^T$ give examples where, in all cases, the third component of $V^\dagger\vec{w}$ and $V^T \vec{w}$ is the largest in modulus, but the smallest growth in loss of orthogonality with $k=1$ projection steps is obtained by projecting $\vw$ against~$\vv_2$.

We conclude that there must be some condition on $V$ (and possibly $\vw$) to guarantee that the ``correct'' (optimal) support of $k$~coefficients is selected. This is a well-known fact in  compressive sensing \cite{foucart2013invitation}, where conditions like the restricted isometry property~(RIP, \cite{candes2005decoding}) are needed to guarantee exact or approximate recovery in sparse least squares approximation (see, e.g., \cite{candes2006stable} or \cite[Thm.~C]{tropp2004greed}). On the other hand, the sensing problem $A\vx = \vb$ studied in this field is usually underdetermined and one often has the freedom to choose the columns of $A$ (the dictionary), whereas in our case we would like to select $k$ columns from a basis which is otherwise unstructured. Moreover, in compressive sensing the vector $\vx$ that one wants to recover is usually sparse, or at least well approximated by a sparse vector, while in our case the coefficient vectors are generally dense. Nevertheless, we believe that results developed in compressive sensing may be used to gain more insight into the sketch-and-select Arnoldi process.

\section{Conclusions and future work}

We have introduced a sketch-and-select Arnoldi process and demonstrated its potential to generate Krylov bases that are significantly better conditioned than those computed with the truncated Arnoldi process, at a computational cost that grows only linearly with the dimension of the Krylov space. 
\revone{The bases generated with a sketch-and-select procedure can be used within sketch-based Krylov solvers for linear algebra problems, such as the sGMRES method for linear systems, and they allow us to run the method for more iterations compared to truncated Arnoldi before the basis condition number becomes too large.}

\revone{We have formulated the problem of generating a well-conditioned Krylov basis as a best subset selection problem for sparse approximation, similar to those encountered in statistical learning and compressive sensing.}
While in principle any of the many  methods that have been proposed for these problems can be employed, we have been surprised to find that the most basic variant of the sketch-and-select Arnoldi process shown in \cref{alg:ssa} is among the best. In this approach we simply retain the $k$ largest modulus coefficients of the least squares solution, setting the remaining coefficients to~zero.

Our implementation of the sketch-and-select Arnoldi process in \cref{alg:ssa} is not optimized for performance, with several straightforward improvements possible including  the use of QR updating strategies for the solution of the least squares problem or performing some of the operations in reduced precision. Also, for a practical implementation to be used in production, the sketch-and-select Arnoldi process could be modified to adapt the parameter~$k$ dynamically based on the measured growth of the condition number. Bounds like the ones derived in \cref{sec:cond} might be useful to control the condition number growth efficiently. Further possible extensions include a sketch-and-select block Arnoldi process or restarting strategies. 

\revtwo{Finally, we note that our initial motivation for this work was to devise a multi-purpose method that would allow for the iterative construction of Krylov bases~$V_m$ with slowly growing condition number. However, our experiments with sGMRES in \cref{sec:ill} indicated that, in some cases, the condition number (i.e., the extremal singular values of $V_m$) may not be the only quantity that affects the numerical behavior of sGMRES. We suspect that  the \emph{distribution} of all singular values plays a role as well, but the numerical results are currently somewhat inconclusive. A dedicated investigation and analysis of sGMRES is needed to explain the effects we observed.}

\section*{Acknowledgments}
I.~S. wishes to thank S.~G. for the kind hospitality received during a research visit to the University of Manchester, where this work was initiated. We thank J{\"o}rg Liesen for communications on the Faber--Manteuffel theorem. \revone{We are grateful to the two anonymous referees for their insightful comments and suggestions.}

\bibliographystyle{siamplain}
\bibliography{manuscript-biblio}

\end{document}